\documentclass[12pt, reqno]{amsart}
\usepackage{amsthm}
\usepackage{amsfonts, ulem}
\usepackage{amsmath}
\usepackage{amssymb}
\usepackage{mathrsfs}
\usepackage{bbm}
\usepackage{tikz, hhline}
\usepackage{multicol}
\usepackage{tabularx, multirow}
\usepackage[utf8]{inputenc}
\usetikzlibrary{matrix}
\usepackage{enumitem}
\usepackage{color}
\usepackage[all]{xy}

\everymath{\displaystyle}
\CompileMatrices

\CompileMatrices

\newtheorem{thm}{Theorem}[section]
\newtheorem{lemma}[thm]{Lemma}
\newtheorem{prop}[thm]{Proposition}

\theoremstyle{definition}
\newtheorem{defn}[thm]{Definition}

\newtheorem{const}[thm]{Construction}

\def\R{{\mathbb R}}

\def\C{{\mathbb C}}
\def\H{{\mathbb H}}

\def\T{{\mathbb T}}
\def\Z{{\mathbb Z}}

\def\Up{{\Upsilon}}

\def\1{{1}}

\def\su{_{\scriptscriptstyle U}}

\def\pr{^{\scriptscriptstyle \R}}
\def\po{^{\scriptscriptstyle O}}
\def\pu{^{\scriptscriptstyle U}}

\def\sc{_{\scriptscriptstyle \C}}
\def\crr{^{\scriptscriptstyle {\it CR}}}

\def\T{ ^{\rm{Tr}}}
\def\sT{ ^{\sharp \otimes {\rm{Tr}}}}
\def\Ttau{ ^{\rm{Tr} \otimes \tau}}
\def\sTtau{^{\sharp \otimes {\rm{Tr}} \otimes \tau}}
\def\stau{^{\sharp \otimes \tau}}

\def\diag{\rm{diag}}

\def\ev{\rm{ev}}

\newcommand{\sm}[4]
	{ \left( \begin{smallmatrix} {#1} & {#2} \\ {#3} & {#4} \end{smallmatrix} \right)  }

\def\id{\rm {id} \,}

\begin{document}

~

\vspace{-1cm}

\title{The Real $K$-Theory of the Sphere with an Arbitrary Involution} 
\author{Jeffrey L. Boersema}
\begin{abstract}
We complete the investigation begun in a previous paper to find unitary representations of the non-trivial real $K$-theory elements for the sphere $S^d$ with an involution. Here we consider all involutions except the antipodal involutions. We write down explicit unitaries representing the generators in all cases for $d \leq 3$, and for $d > 0$ we describe a recipe for generating such unitaries.
\end{abstract}

\maketitle



\section{Introduction}
On the unit sphere $S^d \subset \R^{d+1}$,
there exists a family of nonequivalent involutions $\tau =  \tau^{a,b}$ (where $a + b = d+1$) defined by
$$(x_1, \dots, x_a, x_{a+1} \dots, x_{a+b}) \mapsto (x_1, \dots, x_a, -x_{a+1}, \dots, -x_{a+b}) \; $$
for $x = (x_1, \dots, x_{d+1}) \in S^d$. We are interested in the real $K$-theory of the space-with-involution $(S^d, \tau)$, which is to say the $K$-theory $KO_i(A)$ of the real $C \sp *$-algebra
$$A = C\pr(S^d, \tau) = \{ f \in C(S^d) \mid f(\tau(x)) = \overline{f(x)} \} \subset C(S^d) \;. $$
The $K$-theory of this real $C \sp *$-algebra is naturally isomorphic to the topological ``Real $K$-Theory" $KR^i(S^d, \tau)$ (as in Atiyah \cite{AtiyahR}), 
defined in terms of vector bundles ($\xi, \sigma)$ over $S^d$ which have a conjugate linear involution  $\sigma$ intertwining with $\tau$.
See Section~16 of \cite{BoersemaSchochet} for a proof of this isomorphism.

If $a > 0$ then the involution $\tau^{a,b}$ on $S^d$ has fixed points, and by performing a sterographic projection on one of the fixed points of $S^d$ we find an isomorphism
\begin{equation} \label{AbstractKTheory}
\begin{aligned}
KO_*(C\pr(S^d, \tau^{a,b})) 
&= KO_*(\R) \oplus KO_*(C\pr_0(\R^d, \tau^{a-1,b})) \\
&= KO_*(\R) \oplus \Sigma^{a-b-1} KO_*(\R) 
	\; . 
\end{aligned}
 \end{equation}
 We tend to disregard the first summand of this decomposition, which is associated with the real-valued constant functions on $S^d$. The second summand is isomorphic to the reduced $K$-theory, identified as the kernel of a point-evaluation map
 $$\widetilde{KO}_*(C\pr(S^d, \tau^{a,b})) = \ker \left[ \ev_* \colon {KO}_*(C\pr(S^d, \tau^{a,b}) )
 	\rightarrow KO_*( \R) \right]$$
where $\ev \colon C\pr(S^d, \tau^{a,b}) \rightarrow \R$ is evaluation at any point $p \in S^d$ fixed by the involution $\tau^{a,b}$.
 The reduced $K$-theory contains the information of interest:
\begin{equation*}  
\widetilde{KO}_*(C\pr(S^d,  \tau^{a,b})) \cong KO_*(S^{a-b-1} \R) \cong \Sigma^{a-b-1} KO_*(\R) \; \end{equation*}
where the algebra $S^{a-b-1} \R$ is defined in terms of the suspension or desuspension operations 
\begin{align*}
SA &= C\pr(\R, \tau^{1,0}) \otimes A \\
S^{-1} A &= C\pr(\R, \tau^{0,1}) \otimes A
\end{align*}
have the effect of shifting the real $K$-theory up or down (see Proposition~1.20 of \cite{Boer2002}).

Recall that $KO_*(\R)$ is a period-8 graded ring whose groups are given by
 $$KO_i(\R) =  (\Z, \Z_2, \Z_2, 0, \Z, 0, 0, 0)$$
 in degrees $0 \leq i \leq 7$. Thus from Equation~(\ref{AbstractKTheory}) we have
 $$\widetilde{KO}_i(C\pr(S^d, \tau^{a,b})) = (\Z, \Z_2, \Z_2, 0, \Z, 0, 0, 0)$$
 in degrees $b-a+1 \leq i \leq b-a+8$. We refer to the {\it  key generator} as the generator 
 $$g_{a,b} \in \widetilde{KO}_{b-a+1}(C\pr(S^d, \tau^{a,b})) \cong \Z$$ 
 which generates
 $\widetilde{KO}_*(C\pr(S^d,  \tau^{a,b}))$ as a $KO_*(\R)$-module and furthermore generates the united $K$-theory
 $\widetilde{K}\crr(C\pr(S^d,  \tau^{a,b}))$ as a free $\mathcal{CR}$-module. See Section~\ref{prelim} below for a summary review of united $K$-theory or the references listed there for a complete introduction.  
We note that apart from the exceptional cases $d = 0,1$, the same formula of Equation~(\ref{AbstractKTheory}) holds also for the antipodal involution, even though the stereographic projection argument does not carry through 
(see Theorem~2.1 of \cite{boersema-spheres1}).

Equation~(\ref{AbstractKTheory}) gives us the abstract group structure of $KO_i(C\pr(S^d, \tau^{a,b}))$ for all $i$ and all $a,b$. This paper is concerned with identifying explicit descriptions of the non-trivial $K$-theory elements, using the unitary picture of real $K$-theory developed in \cite{BL} and \cite{Boer2020}.
 Some low-dimensional cases (for $d = 1,2$) of this was accomplished in Section~4 of \cite{Boer2020}; but in \cite{boersema-spheres1} and in this paper, we take a systematic approach to finding unitary elements to represent all non-trivial $K$-theory classes in all cases. We find explicit formulas for this, for all spheres $(S^d, \tau^{a,b})$ with $d \leq 3$ and $a > 0$, and with all possible involutions. More generally, for each $a,b$ with $a > 0$, we will describe a process to find a formula for a unitary $u$ that represents the key generator $g_{a,b}$. The cases where $a = 0$ was fully addressed in \cite{boersema-spheres1}, so together these two papers address all possible spheres with involution. 
  
This paper is also closely connected with the work of Schulz-Baldes and Tom Stoiber in \cite{SBS}, which addresses the case of complex $K$-theory for the sphere $S^d$. There, in Proposition~1, the authors display unitaries that represent the non-trivial elements $g$ in $\widetilde{K}_i(C(S^d))$. In particular,
\begin{align*}
&g \in \widetilde{K}_0(C(S^d)) = \Z &\text{ for $d$ even} \\
&g \in \widetilde{K}_1(C(S^d)) = {K}_1(C(S^d)) = \Z &\text{ for $d$ odd} 
\end{align*}
The formulas that we find for $g_{a,b}$ will be derived from the formulas of \cite{SBS} in the complex case, but extra care must be taken throughout to make sure that they satisfy the required symmetries for the real $K$-theory elements and taking into account the given involution on $S^d$, as dictated by the unitary picture of real $K$-theory.

In the unitary picture of real $K$-theory, elements of real $K$-theory are represented by unitary matrices over the complexification of the real $C \sp *$-algebra in question, that satisfy certain symmetry relations. This picture is summarized in Section~\ref{prelim} below.

In the low dimensional cases $d \leq 3$, we present these formulas for generators of all the non-zero $K$-theory groups, presenting these formulas as concretely as possible, so that they can be used off-the-shelf. This is accomplished in Section~\ref{section-lowdim}. We treat the general case, for arbitrary $d$, in Section~\ref{section-highdim}. We produce formulas of unitaries to represent the key generators $g_{a,b}$. All other non-trivial elements of 
$\widetilde{KO}_{i}(C\pr(S^d, \tau^{a,b}))$ can be obtained from $g_{a,b}$ using the internal $K$-theory operations, arising from its structure as a module over $K_*(\R)$. Accordingly, a dedicated reader can combine the formulas for unitaries representing $g_{a,b}$ in this paper with the concrete description of these operations in \cite{Boer2020}, 
to find a unitary representative for any non-trivial element in $\widetilde{KO}_{i}(C\pr(S^d, \tau^{a,b}))$.

Before describing the most general situation in Section~\ref{section-highdim}, we address two special families of involutions where we can present more specific results. The special families of examples are for the identity involution $\tau^{a, 0}$ and for the involution $\tau^{1, b}$. For these families of examples, we have
\begin{align*}
	\widetilde{KO}_i(C\pr(S^d, \tau^{a,0})) &\cong KO_i(C\pr_0(\R^d, \tau^{a-1,0})) = KO_i(S^{a-1} \R) \; .\\
	\widetilde{KO}_i(C\pr(S^d, \tau^{1,b})) &\cong KO_i(C\pr_0(\R^d, \tau^{0,b})) = KO_i(S^{-b} \R) \; .\\
\end{align*}

We show how to find unitaries for $g_{a,b}$ in these special cases.
Our intention is that readers who are interested in one of these special families of examples will find the description more approachable, and also that the presentation in the special cases will better illuminate the general situation which will follow.

Formulas such as the ones we are developing in this paper are topological markers in physical models such as in \cite{bradlyn},  \cite{Li...2025}, and \cite{Loring2015}. These physical models tend to involve a space with involution $(X, \tau)$ where the involution might represent time-reversal symmetry or particle-hole symmetry (or a combination of both) and the space $X$ might be homeomorphic to $\R^n$, a sphere $S^d$, a torus (a product of spheres), or another space.

These methods have been used to study topological insulators and more recently these methods have been proposed to classify topology in photonic systems as in \cite{CL-2022}, \cite{CL2024}, \cite{CLSB-2024}, \cite{DLC-2023}.
We wish to be able to address physical models in a variety of possible dimensions, with a variety of symmetry structures, which may or may not contain time-reversal symmetry, and which may or may not contain particle-hole symmetry. Together with the previous paper \cite{boersema-spheres1}, this article extends the range of such models that can be addressed to include all spheres with involution. The formulas that we develop in this paper should be seen as $K$-theory generalizations of the index formulas described in \cite{HL}, particularly Section 5.

\subsection{Acknowledgements}

This work was supported in part by the Laboratory Directed Research and Development program at Sandia National Laboratories.
Also, this work was performed, in part, at the Center for Integrated Nanotechnologies, an Office of Science User Facility operated for the U.S. Department of Energy (DOE) Office of Science by Los Alamos National Laboratory (Contract 89233218CNA000001) and Sandia National Laboratories (Contract DE-NA-0003525).

\section{Preliminaries} \label{prelim}

Most of the key preliminary material that we need can be found summarized in more detail in the Introduction of \cite{boersema-spheres1} to which we refer the reader.

A recipe for representing the non-trivial elements of the complex $K$-theory $\widetilde{K}_i(C(S^d))$ for all $d \geq 1$ in terms of self-adjoint unitaries (for $i = 0$ and $d$ even) and unitaries (for $i = 1$ and $d$ odd) can be found in the paper by Schulz-Baldes and Stoiber 
(see Proposition~1 and the beginning of Section~2 of \cite{SBS}).
 Briefly, for a positive odd integer $d$, first write down a set of Clifford generators $\Gamma_{1, d}, \dots \Gamma_{d,d}$, which are matrices in $M_{n}(\C)$ where $n = 2^{{(d-1)}/2}$ that satisfy 
 $$\Gamma_{i, d} \Gamma_{j,d} +  \Gamma_{j, d} \Gamma_{i,d}  = 2 \delta_{i,j} \; .$$
 We obtain a self-adjoint unitary $Q_{d-1} \in C(S^{d-1}) \otimes  M_n(\C)$ and a unitary $U_d \in C(S^{d}) \otimes M_n(\C)$ by
\begin{align} \label{QandU}
	Q_{d-1}(x) &= \sum_{i=1}^{d} x_i \Gamma_{i,d} \\
	 \text{and} \quad 
	U_d(x) &= \sum_{i=1}^d x_i \Gamma_{i,d}  + x_{d+1} i \,  I   \; .  \nonumber
\end{align}
Then the authors of \cite{SBS} prove that the elements
$$[Q_{d-1}] \in \widetilde{K}_0(C(S^{d-1})) = \Z \quad \text{and} \quad
[U_d] \in \widetilde{K}_1(C(S^d) ) = {K}_1(C(S^d) ) = \Z$$ represent generators of the respective $K$-theory groups.
In this expression, complex $K_0$ elements are represented using self-adjoint unitaries and complex $K_1$ elements are represented using arbitrary unitaries in matrix algebras over the complex $C \sp *$-algebra $C(S^d)$, as shown in the first lines of Table~\ref{unitaryTable}.

In the real case, we also represent $KO_i$ elements using either self-adjoint unitaries or arbitrary unities in matrices algebras over the complex $C \sp *$-algebra $C(S^d)$ that also satisfy an extra symmetry related to the anti-multiplicative involution on $C(S^d)$ (coming from the involution on $S^d$). The symmetry that must be satisfied varies according to $i$ as laid out in the unitary picture of real $K$-theory developed in \cite{BL} and summarized here in Table~\ref{unitaryTable}.

\newcommand\TT{\rule{0pt}{2.6ex}} 
\newcommand\BB{\rule[-1.2ex]{0pt}{0pt}} 
\begin{table}[h] 
\caption{$KO_*(A)$ and $KU_*(A)$ via unitaries for $A$ unital.} \label{unitaryTable}
\begin{center}
\begin{tabular}{c|c|c|c|c|} 
\hhline{~----}
& K-group \TT \BB & $n_i$ & $\mathscr{S}_i$ & $I^{(i)}$ \\ \hhline{=====}
\multicolumn{1}{|c|}{\multirow{2}{*}{complex}}  & $KU_0(A)$ \TT \BB & 2 & $u = u^*$ & $\sm{\1}{0}{0}{-\1}$ \\ \hhline{~----}
\multicolumn{1}{|c|}{} & $KU_1(A)$ \TT \BB & 1 & -- & $\1$ \\ \hhline{=====}
\multicolumn{1}{|c|}{\multirow{8}{*}{real}}  & $KO_{-1}(A)$ \TT \BB & 1 & $u\Ttau = u $ & $\1$ \\ 
\multicolumn{1}{|c|}{} & \text{or} & { 2}  & {  $u\sTtau = -u$} & { $\sm{0}{\1}{-\1}{0}$ }\\
 \hhline{~----}
\multicolumn{1}{|c|}{} & $KO_0(A)$ \TT \BB & 2 &  $u = u^*$, $u\Ttau = u $ & $\sm{\1}{0}{0}{-\1}$  \\ \hhline{~----}
\multicolumn{1}{|c|}{} & $KO_1(A)$ \TT \BB & 1 &  $u\Ttau = u^* $ & $\1$ \\ \hhline{~----}
\multicolumn{1}{|c|}{} & $KO_2(A)$ \TT \BB & 2 &  $u = u^*$, $u\Ttau = -u$ & $\sm{0}{i \cdot \1}{-i \cdot \1}{0}$ \\\hhline{~----}
\multicolumn{1}{|c|}{} & $KO_{3}(A)$ \TT \BB & 2 & $u\sTtau = u$ & $ \1_2 $ \\
\multicolumn{1}{|c|}{} & \text{or} & 2 &  $u\Ttau = -u$ & { $\sm{0}{\1}{-\1}{0}$ }\\
 \hhline{~----}
\multicolumn{1}{|c|}{} & $KO_4(A)$ \TT \BB & 4 &  $u = u^*$, $u\sTtau = u$ & ${\diag}(\1_2,-\1_2)$ \\ \hhline{~----}
\multicolumn{1}{|c|}{} & $KO_5(A)$ \TT \BB & 2 &  $u\sTtau = u^*$ & $\1_2$ \\ \hhline{~----} 
\multicolumn{1}{|c|}{} & $KO_6(A)$ \TT \BB & 2 &  $u = u^*$, $u\sTtau = -u$ & $\sm{0}{i \cdot \1}{-i \cdot \1}{0}$ \\ \hhline{=====}
\end{tabular}
\end{center}
\end{table}

To understand this table, recall that for any real $C \sp *$-algebra $A$ there is a corresponding antimultiplicative involution $\tau$ on $A\sc = A \otimes \C$, given by
$\tau \colon a + ib \mapsto a^* + ib^*$ where $a,b \in A$. The real $C \sp *$-algebra $A$ can be recovered from the pair $(A\sc, \tau)$ as the elements in $A\sc$ that satisfy $a^\tau = a^*$. The involution $\tau$ extends to involutions $\rm{Tr} \otimes \tau$ on $M_n(\C) \otimes A\sc = M_n(A\sc)$ and involutions $\sharp \otimes \rm{Tr} \otimes \tau$ on $M_2(\C) \otimes M_n(\C) \otimes A\sc = M_{2n}(A\sc)$.
The involution $\sharp$ on $M_2(\C)$ is the involution associated with the algebra of quaternions $\H \subset M_2(\C)$ and is given by
$$\begin{bmatrix} a & b \\ c & d \end{bmatrix}^\sharp = \begin{bmatrix} d & -b \\ -c & a \end{bmatrix} \; .$$
When we are making explicit calculations with matrices, we always choose an isomorphism $M_2(\C) \otimes M_n(\C) \cong M_{2n}(\C)$ such that
	$$a^{\sharp \otimes \rm{Tr}}
		= \begin{bmatrix} a_{1,1} & a_{1,2} & \cdots & a_{1,n} \\ a_{2, 1} & a_{2,2} & \cdots & a_{2,n} \\
			\vdots & \vdots & \ddots & \vdots \\ a_{n,1} & a_{n,2} & \cdots & a_{n,n}  \end{bmatrix}^{\sharp \otimes \rm{Tr}} 
	 	= 	\begin{bmatrix} a_{1,1}^{\sharp} & a_{2,1}^\sharp & \cdots & a_{n,1}^\sharp \\
			 \\ a_{1,2}^\sharp & a_{2,2}^\sharp & \cdots & a_{n,2}^\sharp \\
			\vdots & \vdots & \ddots & \vdots \\ a_{1,n}^\sharp & a_{2,n}^\sharp & \cdots & a_{n,n}^\sharp \end{bmatrix}
		$$
where each $a_{i,j} \in M_2(\C)$.
		
For a unital real $C \sp *$-algebra, the elements of $KO_i(A)$ are represented by unitaries in $M_{n}(A\sc)$ that satisfy the relation $\mathscr{S}_i$ given in the table (and where $n$ is a multiple of $n_i$). The designated unitary element $I^{(i)}$ in the table represents a neutral element, which satisfies $[ I^{(i)} ] = 0$ in $KO_i(A)$.  The identification $[u] = [\diag(u, I^{(i)})] \in KO_i(A)$ shows how to identify a unitary in $M_n(A\sc)$ with a unitary in $M_{n+n_i}(A\sc)$ representing the same $KO$-class. Note that in the cases for $i = 0$ and $i = 1$, the symmetry $u\Ttau = u^*$ is equivalent to saying that $u$ lives in the real $C \sp *$-algebra $A$ (or a matrix algebra over $A$).

We also note that there are two lines in the table for each of $KO_{-1}(A)$ and $KO_3(A)$. The first line represents the picture of $KO_i(A)$ developed in \cite{BL} and also used in \cite{Boer2020}. The second line represents an alternative variation of the unitary picture of $KO_i(A)$ developed and described in the Appendix of \cite{boersema-spheres1} (the results there also include specific descriptions of the formulas for the homomorphisms from one version to the other).
We will use both versions as convenient.

Throughout, we will also use some of the language of united $K$-theory (introduced in \cite{bousfield90}, \cite{Boer2002} and summarized for these purposes in \cite{boersema-spheres1}). Briefly, for a real $C \sp *$ algebra $A$, we consider the period-8  $K$-theory $KO_*(A)$ along with the period-2 $K$-theory of the complexification $KU_*(A) = K_*(A\sc)$. This package $K\crr(A) = \{ KO_*(A), KU_*(A) \}$
also comes with a collection of natural transformations among them. In particular, the natural transformations $r,c, \eta$ contribute to the long exact sequence 
\begin{equation}
\label{eq:CR-LES}
\dots \rightarrow
KO_{i}( A) \xrightarrow{\eta} 
KO_{i+1}( A) \xrightarrow{c} 
KU_{i+1}( A) \xrightarrow{r \beta\su^{-1}} 
KO_{i-1}(A) \rightarrow
\cdots  \; .
\end{equation}
which we find indispensible. This ensemble of information is called a $\mathcal{CR}$-module. For us the important thing to know is that $K\crr( C\pr(S^d, \tau^{a,b}))$ is a free $\mathcal{CR}$-module generated by the key generator $g_{a,b} \in KO_{b-a+1}( C\pr(S^d, \tau^{a,b}))$.

The natural transformations of united $K$-theory are defined by an underlying $C \sp *$-algebra homomorphism (or in some cases, by a $KK$-element). For example
$c \colon KO_*(A) \rightarrow KU_*(A)$ is induced by the natural inclusion $A \hookrightarrow A\sc$; and 
$r \colon KU_*(A) \rightarrow KO_*(A)$ is induced by the natural inclusion $A\sc \rightarrow M_2(A)$. 
Moreover, in \cite{Boer2020} we developed explicit descriptions of all of these natural transformations in terms of the unitary picture of $K$-theory. We will use these formulas frequently in the following. Furthermore, in the higher-dimensional cases, in Section~\ref{section-highdim}, we content ourselves with finding the unitary representations of the key generator $g_{a,b}$, knowing that all other elements of $KO_*(C\pr(S^d, \tau^{a,b}))$ can be derived from $g_{a,b}$ using the natural transformations.

For reference, we show in Table~\ref{Table-freemodule} the structure of the free $\mathcal{CR}$-module $M = (M\po_*, M\pu_*)$ generated by a single element $g \in M\po_0$.  
\begin{table}[b]  
\caption{The free $\mathcal{CR}$-module $M$} \label{Table-freemodule}
\[ \begin{array}{|c|c|c|c|c|c|c|c|c|c|}  
\hline \hline  
n & ~0~ & ~1~ & ~2~ & ~3~ & ~4~ & ~5~ & ~6~ & ~7~  \\
\hline  \hline
M\po_n
& \Z  & \Z_2  & \Z_2  & 0 
	& \Z  & 0 & 0 & 0  \\
\hline  
M\pu_n 
& \Z  & 0 & \Z  & 0 & \Z   & 0 
	& \Z  & 0   \\
\hline \hline
\eta_n & 1 & 1 & 0 & 0 & 0 & 0 & 0 & 0 \\
\hline 
c_n & 1 & 0 & 0 & 0 & 2 & 0 & 0 & 0   \\
\hline
r_n & 2 & 0 & 1 & 0 & 1 & 0 & 0 & 0      \\
\hline
\hline \hline
\end{array} \]
\end{table}
Note that $K\crr(\R) \cong M$ and that $K\crr(S^d \R) = \Sigma^d M$ for any integer $d$.
We will use this structure frequently in the calculations of the following sections. Note that $c_0 \colon M\po_0 \rightarrow M\pu_0$ is an isomorphism
and that $r_i \colon M\pu_i \rightarrow M\pu_i$ are surjections for $i = 2,4$.


\section{Unitary Generators in Low Dimensions} \label{section-lowdim}

In this section, we present explicit unitary generators for all the $K$-theory groups low dimensional sphere. The goal here is to have  formulas ready off-the-shelf in the low-dimensional cases that can be used for topological localizers in mathematical physics. We present the formulas for generators of all non-zero groups $KO_i(C\pr(S^d, \tau^{a,b}))$ where $d = 1,2,3$ and $a \neq 0$. Some of these cases appear previously in \cite{Boer2020}, which we restate here without proof for completeness.

\subsection{All 1-spheres}

First we present unitary generators for all the $K$-theory groups for the circle $S^1$, either with the trivial involution $\tau^{2,0}$ or with the involution $\tau^{1,1}$ given by $(x,y) \mapsto (x,-y)$ on $S^1$. Recall that
\begin{align*} \widetilde{KO}_*(C\pr(S^1, \tau^{1,1})) &= \Sigma^{-1} KO_*(\R) \\
\text{and } \widetilde{KO}_*(C\pr(S^1, \tau^{2,0})) &= \Sigma KO_*(\R) \; 
\end{align*}
so the key generators $g_{a,b}$ generating each $\mathcal{CR}$-module are in degree $1$ and in degree $-1$, respectively.

The results shown in Tables~\ref{S1unitary1} and \ref{S1unitary2} are from Theorem~4.2 and Theorem~4.3 of \cite{Boer2020}. 

\begin{table}[h]
\caption{Unitaries for $\widetilde{K}\crr(C\pr(S^1, \tau^{1,1}))$.} \label{S1unitary1}
$$\begin{array}{|c|c|c|}  \hline
   (x,y) \mapsto (x,-y) & \text{isomorphism class} & 
 	\multicolumn{1}{|c|}{ \text{unitary representing a generator}}  \\ \hline \hline
\widetilde{KU}_0 & 0 &  \\ \hline
\widetilde{KU}_1 & \Z & y_1 = x+iy  \\ \hline  \hline
\widetilde{KO}_0 & 0 &  \\ \hline
\widetilde{KO}_1  & \Z  & x_1 = x+iy  \\ \hline 
\widetilde{KO}_2 & \Z_2  & x_2 = \sm{y}{ix}{-ix}{y} \\ \hline 
\widetilde{KO}_3  & \Z_2 & x_3 = \sm{x+iy}{0}{0}{x-iy} \\ \hline 
\widetilde{KO}_4  & 0  &  \\ \hline 
\widetilde{KO}_5  & \Z & x_5 = \sm{x+iy}{0}{0}{x+iy}  \\ \hline 
\widetilde{KO}_6  & 0 & \\ \hline 
\widetilde{KO}_7  & 0 & \\  \hline 
\end{array}$$
\end{table}

\begin{table}
\caption{Unitaries for $\widetilde{K}\crr(C\pr(S^1, \tau^{2,0}))$.} \label{S1unitary2}
$$\begin{array}{|c|c|c|}  \hline
  (x,y) \mapsto (x,y) & \text{isomorphism class} & 
 	\multicolumn{1}{|c|}{ \text{unitary representing a generator}}  \\ \hline \hline
 \widetilde{KU_0} & 0  &  \\ \hline
 \widetilde{KU_1} & \Z  & y_1 = x + iy  \\ \hline \hline
\widetilde{KO}_{-1} & \Z & x_{-1} = x + iy   \\ \hline 
\widetilde{KO}_0 & \Z_2 & x_{0} = \sm{x}{y}{y}{-x}  \\ \hline 
\widetilde{KO}_1 & \Z_2 & x_{1} =  \sm{x}{-y}{y}{x}  \\ \hline
\widetilde{KO}_2  & 0 &  \\ \hline 
\widetilde{KO}_3 & \Z & x_{3} =  \sm{x+iy}{0}{0}{x+iy}  \\ \hline
\widetilde{KO}_4 & 0 & \\ \hline 
\widetilde{KO}_5  & 0 & \\ \hline 
\widetilde{KO}_6 & 0 & \\ \hline 
\end{array}$$
\end{table}

\subsection{All 2-spheres}

Here we find unitary generators for following, which represent all involutions (except the antipodal involution) on the sphere $S^2$.
Recall that
\begin{align*} 
\widetilde{KO}_*(C\pr(S^2, \tau^{3,0})) &= \Sigma^2 KO_*(\R) \,  \; \\
\widetilde{KO}_*(C\pr(S^2, \tau^{2,1})) &=  KO_*(\R) \, ,  \\
\text{and~}   	 \widetilde{KO}_*(C\pr(S^2, \tau^{1,2})) &=  \Sigma^{-2} KO_*(\R) \, , \; \\
\end{align*}
so the key generators will be in degrees $-2,0,2$ respectively.

These results are recorded in Tables~\ref{S2unitaryA}, \ref{S2unitaryB}, and \ref{S2unitaryC}. For the first two of these tables, these result are from Theorem 4.4 and Theorem 4.5 of \cite{Boer2020}. For Table \ref{S2unitaryC}, we state and prove this result for the first time here in Proposition~\ref{S2:1,2}.

\begin{table}[h]
\caption{Unitaries for $\widetilde{K}\crr(C\pr(S^2, \tau^{3,0} ))$.} \label{S2unitaryA}
$$\begin{array}{|c|c|c|}  \hline
 (x,y,z) \mapsto (x,y,z) & \text{isomorphism class} 
 	 & { \text{unitary representing a generator}}  \\ \hline \hline
 \widetilde{KU}_0  & \Z & y_0 = \sm{x}{y-iz}{y+iz}{-x}  \\ \hline
 \widetilde{KU}_1  & 0 &   \\ \hline \hline
\widetilde{KO}_{-2}  & \Z & x_{-2} = \sm{x}{y-iz}{y+iz}{-x} \\ \hline
\widetilde{KO}_{-1}  & \Z_2 & x_{-1} = \sm{x+iy}{iz}{iz}{x-iy} \\ \hline
\widetilde{KO}_{0} & \Z_2 &  x_0 = \left(
\begin{smallmatrix}
x & 0 & y & -z \\
0 & x & z & y \\
y & z & -x & 0 \\
-z & y & 0 & -x
\end{smallmatrix} \right)     \\ \hline  
\widetilde{KO}_{1}  & 0 &    \\ \hline 
\widetilde{KO}_{2} & \Z & x_2 =  i \left(
 \begin{smallmatrix}
0 & x & z & y \\ -x & 0 & y & -z \\ -z & -y & 0 & x \\ -y & z & -x & 0
\end{smallmatrix} \right) \\ \hline 
\widetilde{KO}_{3}& 0 &   \\ \hline 
\widetilde{KO}_{4} & 0 & \\ \hline
\widetilde{KO}_{5} & 0 &     \\  \hline 
\end{array}$$
\end{table}

\begin{prop} \label{S2:1,2}
Table \ref{S2unitaryC} shows unitary representatives of generators for $\widetilde{K}\crr(C\pr(S^2,  \tau^{1,2}))$.
\end{prop}

\begin{proof}
The involution $\tau = \tau^{1,2}$ on $S^2$ is given on the coordinates by $(x,y,z) \mapsto (x,-y,-z)$.
We make use of the Clifford generators (the Pauli matrices)
$$\sigma_1 = \begin{pmatrix} 0 & 1 \\ 1 & 0 \end{pmatrix} \; , \quad
\sigma_2 =  \begin{pmatrix} 0 & i \\ -i & 0 \end{pmatrix} \; , \quad
\sigma_3 =  \begin{pmatrix} 1 & 0 \\ 0 & -1 \end{pmatrix} \; $$
so the self-adjoint unitary
$$y_0(x,y,z) =  y \sigma_1 + x \sigma_2 + z \sigma_3 = \sm{z}{y+ix}{y-ix}{-z}$$
is a generator of $\widetilde{K}_0(C(S^2)) = \widetilde{KU}_0(C\pr(S^2, \tau^{1,2}))$ by Proposition~1 of \cite{SBS}.

Now, setting 
$$x_2(x,y,z) =  \sm{z}{y+ix}{y-ix}{-z}$$
we check that $x_2$ is a self-adjoint unitary that satisfies $u\Ttau = -u$. Indeed,
$$x_2\Ttau = \sm{z}{y+ix}{y-ix}{-z} \Ttau
	= \sm{z^\tau}{(y-ix)^\tau}{(y+ix)^\tau}{-z^\tau} = \sm{-z}{-y-ix}{-y+ix}{z} = -x_2\; .$$
We also check that
\begin{center}  $\ev_*[x_2] = [\ev(x_2)] = [\sm{0}{i}{-i}{0}] = [ I^{(2)}  ] = 0$ in $KO_2(\R)$.\end{center}
where $\ev$ is evaluation at the point $(1,0,0,0)$.
Thus $[x_2] \in \ker(\ev_*)$
and from this is follows that $[x_2]$ is a legitimate element of $\widetilde{KO}_2(C\pr(S^2, \tau^{1,2}))$.

We now argue that $[x_2]$ is in fact a generator of $\widetilde{KO}_2(C\pr(S^2, \tau^{1,2})) = \Z$.
The complexification map 
$$c_{n} \colon \widetilde{KO}_{n}(A) \rightarrow \widetilde{KU}_{n}(A)$$ is defined in the unitary picture of $K$-theory by forgetting the extra symmetry that a self-adjoint unitary representing real $K$-theory must satisfy (see Theorem~3.1 of \cite{Boer2020}). 
In this case at hand we see that $c_2[x_2] = [y_0]$. Also, 
for $A = C\pr(S^2, \tau^{1,2})$, 
we know a priori that $c_{2}$ is an isomorphism, because of the structure of the free $\mathcal{CR}$-module, namely that
$\widetilde{K}_*\crr(C\pr(S^2, \tau^{1,2}))$ is a free $\mathcal{CR}$-module with generator in $\widetilde{KO}_{2}(C\pr(S^2, \tau^{1,2})) = \Z$ (because of the suspension, this corresponds to the column for $n = 0$ in Table~\ref{Table-freemodule}).
Since $c_2$ maps $[x_2]$ to a generator of $\widetilde{KU}_0(C\pr(S^2, \tau^{1,2}))$, this shows that $[x_2]$ must be a generator of $\widetilde{KO}_{2}(C\pr(S^2, \tau^{1,2})) = \Z$. 
 
 Since $\eta_2$ is surjective, it must be that $\eta_2[x_2]$ is a generator of $\widetilde{KO}_{3}(C\pr(S^2, \tau^{1,2})) = \Z_2$. The formula from Theorem~5.1 of \cite{Boer2020} is given by
 $$\eta_2\left[ \sm{a}{b}{c}{d} \right] = \sm{ic}{d}{a}{-ib} \; .$$
 We apply this to $x_2$ to find that a formula for $x_3$ is
 $$x_3 = \sm{x+iy}{-z}{z}{x-iy}\; .$$
 (We can double-check that $x_3$ actually represents an element of $KO_3(C\pr(S^2, \tau^{1,2}))$ by verifying
  that $x_3\stau = x_3$.)

To find a formula for $x_4$, we use the fact that $r_4\colon \widetilde{KU}_4(A) \rightarrow \widetilde{KO}_4(A)$ is an surjection for $A = C\pr(S^2, \tau^{1,2})$ (see the column for $n = 2$ in Table~\ref{Table-freemodule}). The formula for $r_4$ from Theorem 3.2 of \cite{Boer2020} is
$$r_4(u) = \sm{a}{ib}{-ib}{a}$$
where 
$$a = \tfrac{1}{2}\left(u + u\stau\right) \quad \text{and} \quad b = \tfrac{1}{2}\left( u - u\stau \right) \; .$$
For us, we use 
$$u = y_0 = \sm{z}{y + ix}{y - ix}{-z}$$ 
and calculate
$$u^{\sharp \otimes \tau} = \sm{z}{y + ix}{y-ix}{-z} \, \quad
	a = \sm{z}{y}{y}{-z} \, \quad
	b = \sm{0}{ix}{-ix}{0} \; .$$
From this we obtain the formula for $x_4$ given in Table~\ref{S2unitaryC}.

Finally, to find a formula for $x_6$ we use the fact that 
$r_6 \colon \widetilde{KU}_6(A) \rightarrow \widetilde{KO}_6(A)$ is an isomorphism for $A = C(S^2, \tau^{1,2})$. 
The formula for $r_6$ (from Theorem~3.2 of \cite{Boer2020}) is given by
$$r_6[u] = \left[ \begin{pmatrix} a & ib \\ -ib & a \end{pmatrix} \right]$$
where $a = \tfrac{1}{2}(u - u\sTtau)$ and $b = \tfrac{1}{2}(u + u\sTtau)$.
Applying this to $u = y_0 = y_6$ gives us
$a = \sm{0}{ix}{-ix}{0}$ and $b = \sm zyy{-z}$ and then we obtain the formula for $x_6$ in Table~\ref{S2unitaryC}.
\end{proof}

We note that a formula for an alternative version of $x_6$ to the one in Table~\ref{S2unitaryC} (and perhaps a preferable one) is given by
$$x_6 = \left( \begin{smallmatrix}
z & 0 & y+ix & 0 \\ 0 & z & 0 & y+ix \\ y-ix & 0 & -z & 0  \\ 0 & y-ix & 0 & -z 
\end{smallmatrix} \right) \; .$$
This version is obtained applying the natural transformation $\xi_2 \colon \widetilde{KO}_2(A) \rightarrow \widetilde{KO}_6(A)$ which is known to be an isomorphism and using the formula for $\xi_2$ in Theorem~5.2 of \cite{Boer2020}. It is not immediately obvious that these two unitaries represent the same element of $\widetilde{KO}_6(A)$, but they must since the relation $r \beta\su^2 c = \xi$ always holds (where $\beta\su$ is the complex Bott periodicity isomorphism).

%

\vspace{.2cm}

\begin{table}[h]
\caption{Unitaries for $\widetilde{K}\crr(C\pr(S^2, \tau^{2,1}))$.} \label{S2unitaryB}
$$\begin{array}{|c|c|c|}  \hline
 (x,y,z) \mapsto (x,y,-z)  & \text{isomorphism class} 
 	 & { \text{unitary representing a generator}}  \\ \hline \hline
 \widetilde{KU}_0 & \Z & y_0 = \sm{x}{y-iz}{y+iz}{-x}   \\ \hline
  \widetilde{KU}_1 & 0 &   \\ \hline \hline
\widetilde{KO}_0 & \Z & x_{0} = \sm{x}{y-iz}{y+iz}{-x} \\ \hline
\widetilde{KO}_1  & \Z_2  & x_1 =\sm{x}{-y + iz}{y+iz}{x} \\ \hline
\widetilde{KO}_2  & \Z_2 & x_2 =  \left( \begin{smallmatrix}
0 & 0 & ix & z + iy   \\ 0 & 0 & -z + iy & -ix \\ -ix & -z - iy & 0 & 0 \\  z - iy & ix & 0 & 0 
\end{smallmatrix} \right)   \\ \hline 
\widetilde{KO}_3  & 0 &    \\ \hline 
\widetilde{KO}_4 & \Z & x_4 = \left( \begin{smallmatrix}
x & 0 & y-iz &  0 \\ 0 & x & 0 & y- iz \\ y + iz & 0 & -x & 0\\ 0 &  y + iz  & 0 & -x
\end{smallmatrix} \right)    \\  \hline 
\widetilde{KO}_5 & 0 &   \\ \hline 
\widetilde{KO}_6 & 0 & \\ \hline
\widetilde{KO}_7  & 0 &     \\  \hline 
\end{array}$$
\end{table}

\vspace{.2cm}

\begin{table}[h]
\caption{Unitaries for $\widetilde{K}\crr(C\pr(S^2,  \tau^{1,2}))$.} \label{S2unitaryC}
$$\begin{array}{|c|c|c|}  \hline
  (x,y,z) \mapsto (x,-y,-z) & \text{isomorphism class} 
 	 & { \text{unitary representing a generator}}  \\ \hline \hline
 \widetilde{KU}_0 & \Z & y_0 = \sm{z}{y+ix}{y-ix}{-z}    \\ \hline
 \widetilde{KU}_1 & 0 &   \\ \hline \hline
 \widetilde{KO}_0 & 0 & \\ \hline
 \widetilde{KO}_1 & 0  &   \\ \hline 
  \widetilde{KO}_2 & \Z & x_2 = \sm{z}{y+ix}{y-ix}{-z} \\ \hline 
 \widetilde{KO}_3 & \Z_2 & x_3 =\sm{x+iy}{-z}{z}{x - iy} \\ \hline 
 \widetilde{KO}_4 & \Z_2 & x_4 = \left(
\begin{smallmatrix}
z & y & 0 & -x \\ y & -z & x & 0 \\ 0 & x & z &y \\ -x & 0 & y & -z
\end{smallmatrix} \right)     \\ \hline 
  \widetilde{KO}_5 & 0 &   \\ \hline 
 \widetilde{KO}_6  & \Z & x_6 = i  \left( \begin{smallmatrix}
 0 & x & z & y \\ -x & 0 & y & -z \\ -z & -y & 0 & x \\ -y & z & -x & 0
\end{smallmatrix} \right)    \\  \hline 
 \widetilde{KO}_7 & 0 & \\ \hline \hline
\end{array}$$
\end{table}

\normalsize

\subsection{All 3-spheres}

Here we find unitary generators for 
\begin{align*} 
		\widetilde{KO}_*(C\pr(S^3, \tau^{4,0})) &= \Sigma^{3} KO_*(\R) \, ,  \\
	 	\widetilde{KO}_*(C\pr(S^3, \tau^{3,1})) &=  \Sigma^{1} KO_*(\R) \, , \; \\
	  	\widetilde{KO}_*(C\pr(S^3, \tau^{2,2})) &=  \Sigma^{-1} KO_*(\R) \, , \; \\
\text{and~} \widetilde{KO}_*(C\pr(S^3, \tau^{1,3})) &= \Sigma^{-3} KO_*(\R) \, , \; \\
\end{align*}
and the key generators will be in degrees $-3,-1,1, 3$ respectively.

These results are recorded in Tables~\ref{S3unitaryA}, \ref{S3unitaryB}, \ref{S3unitaryC}, and \ref{S3unitaryD} below and the proofs for all of these are included.

\begin{table}[h]
\caption{Unitaries for $\widetilde{K}\crr(C\pr(S^3, \tau^{4,0}))$.} \label{S3unitaryA}
$$\begin{array}{|c|c|c|}  \hline
 (x,y,z, w) \mapsto (x,y,z, w)  & \text{isomorphism class} & { \text{unitary representing a generator}}  \\ \hline \hline
 \widetilde{KU}_0 & 0 &  \\ \hline
 \widetilde{KU}_1 & \Z &   y_1 = \sm{x + iw}{y - iz}{y + iz}{-x + iw}       \\ \hline \hline
\widetilde{KO}_0 & 0 & 				\\ \hline 
\widetilde{KO}_1 &  \Z &  x_1 = 
	\left( \begin{smallmatrix}
			x & y & w & -z \\ y & -x & z & w \\  -w & z & x & y \\ -z & -w & y & -x
	\end{smallmatrix} \right)  				\\ \hline 
\widetilde{KO}_2  & 0 &  				\\ \hline
\widetilde{KO}_3 & 0 & 				 \\ \hline 
\widetilde{KO}_4  & 0 & 				 \\ \hline
\widetilde{KO}_5  & \Z & 	
	x_5 = \sm{w - ix}{-z-iy}{z-iy}{w+ix}	 \\ \hline 
\widetilde{KO}_6 & \Z_2 & 		
	x_6 = \left( \begin{smallmatrix}
			0 & 0 & x + iw & y - iz \\ 0 & 0 & y + iz & -x + iw \\ x - iw & y - iz & 0 & 0 \\ y + iz & -x - iw & 0 & 0
		\end{smallmatrix} \right)		\\ \hline 
\widetilde{KO}_7 & \Z_2 & x_7 = 
	\left( \begin{smallmatrix}
			x + iw & y & 0 & z \\ y & -x + iw & -z & 0 \\  0 & -z & x + iw & y \\ z & 0 & y & -x + iw					
	\end{smallmatrix} \right)  				\\ \hline 
\end{array}$$
\end{table}

\begin{prop}
Table \ref{S3unitaryA} shows unitary representatives of generators for $\widetilde{K}\crr(C\pr(S^3, \tau^{4,0}))$.
\end{prop}

We note here we could apply any permutation of the variables $x,y,z,w$ in these formulas or replace any of the variables with its negative, and we would still have a generator of the appropriate group, as these transformations all arise from self-homeomorphisms on $S^3$ which induce automorphisms on $K$-theory. Also note that $[y_1] = [\lambda y_1]$ in $\widetilde{K}_1(C\pr(S^3, \C))$ for any unital complex scalar $\lambda$; and more generally $[y_1] = [u y_1]$ for any unitary $u \in M_2(\C)$.

\begin{proof}
Again we use the Clifford generators $\sigma_1, \sigma_2, \sigma_3$ from the proof of Proposition~\ref{S2:1,2}
and set $$y_1(x,y,z,w) =  y \sigma_1 + x \sigma_2 + z \sigma_3 + w  i I_2 = \sm{x+iw}{y-iz}{y+iz}{-x+iw} \; , $$
so $[y_1]$ is a generator of $\widetilde{K}_1(C(S^3))$ by Proposition~1 of \cite{SBS}.

Now, let 
$x_5$ be as given the table and check that $x_5\stau = x_5^*$ so that $[x_5]$ represents an element of
$\widetilde{KO}_5(C\pr(S^3, \tau^{4,0}))$. Also, note that $x_5$ is the same as $-i y_1$ so it follows from this that $c_5([x_5]) = [-i y_1] = [y_1]$.
Since 
$$c_5 \colon \widetilde{KO}_5(C\pr(S^3, \tau^{4,0})) \rightarrow \widetilde{KU}_5(C\pr(S^3, \tau^{4,0})) \, $$
is an isomorphism,
it follows that $[x_5]$ must be a generator of $\widetilde{KO}_5(C\pr(S^3, \tau^{4,0}))$. This is the key generator of
$\widetilde{K}\crr(C\pr(S^3, \tau^{4,0}))$.

We know that $\eta_5[x_5]$ is the non-zero element of $\widetilde{KO}_6(C\pr(S^3, \tau^{4,0}))$ and from Theorem~5.1 of \cite{Boer2020}, 
we know that $\eta_5$ has the formula
$$\eta_5[u] = \sm{0}{iu}{-iu^*}{0} \; .$$
From this we immediately get the formula for $x_6$ in Table~\ref{S3unitaryA}.

We also know that $r_{7}[y_1]$ is a generator of $\widetilde{KO}_{-1}(C\pr(S^3, \tau^{4,0}))$ and from Theorem~3.2 of \cite{Boer2020}, 
we know that $r_{7}$ has the formula
$$r_{7}[u] = \sm{a}{ib}{-ib}{a} \; $$
where $a = \tfrac{1}{2}(u + u^\tau)$ and $b = \tfrac{1}{2}(u - u^\tau)$.
From this we get the formula for $x_{7}$ in Table~\ref{S3unitaryA}.

We obtain the formula for $x_1$ in the same way, knowing that $r_1[x_1]$ must be a generator of $\widetilde{KO}_{1}(C\pr(S^3, \tau^{4,0}))$.
The formula for $r_1$ from Theorem~3.2 of \cite{Boer2020} is
$$r_{1}[u] = \sm{a}{ib}{-ib}{a} \; $$
where $a = \tfrac{1}{2}(u + u^{*\tau})$ and $b = \tfrac{1}{2}(u - u^{*\tau})$.
\end{proof}

\begin{table}[h]
\caption{Unitaries for $\widetilde{K}\crr(C\pr(S^3, \tau^{3,1}))$.} \label{S3unitaryB}
$$\begin{array}{|c|c|c|}  \hline
 (x,y,z, w) \mapsto (x,y,z, -w)  & \text{isomorphism class} & { \text{unitary representing a generator}}  \\ \hline \hline
 \widetilde{KU}_0 & 0 &  \\ \hline
\widetilde{KU}_1 & \Z &   y_1 = \sm{x + iz}{y + iw}{y - iw}{-x + iz}		       \\ \hline \hline
\widetilde{KO}_{-1}  & \Z & 	x_{-1} = \sm{x + iz}{y + iw}{y - iw}{-x + iz}						\\ \hline 
\widetilde{KO}_0 & \Z_2 &  x_0 = \left( \begin{smallmatrix}
			x & y + iw & z & 0 \\ y - iw & -x & 0 & z \\ z & 0 & -x & -y - iw \\ 0 & z & -y + iw & x
		\end{smallmatrix} \right)		\\ \hline  
\widetilde{KO}_1  &  \Z_2 &  	x_1 = 	\left( \begin{smallmatrix}
			x & -y  - iw & z & 0 \\ y - iw & x & 0 & -z \\ z & 0 & -x & y + iw \\ 0 & -z & -y + iw & -x
		\end{smallmatrix} \right)		\\ \hline  
\widetilde{KO}_2  & 0 &  				\\ \hline
\widetilde{KO}_3 & \Z & 	x_3 = \left( \begin{smallmatrix}
		z & w & x & y \\ -w & z & y & -x \\ -x & -y & z & w \\ -y & x & -w & z	
		\end{smallmatrix} \right)	\\ \hline  
\widetilde{KO}_4 & 0 & 				 \\ \hline
\widetilde{KO}_5 & 0 & 				 \\ \hline 
\widetilde{KO}_6  & 0 & 				\\ \hline 
\end{array}$$
\end{table}

\begin{prop}
Table \ref{S3unitaryB} shows unitary representatives of generators for $\widetilde{K}\crr(C\pr(S^3, \tau^{3,1}))$.
\end{prop}

\begin{proof}
The expression for $y_1$ here is the same as for Table~7 (except for rearranging the variables).
We check that $x_{-1}$ satisfies $x_{-1}^\tau = x^\tau$. Then $[x_{-1}]$ is an element of $\widetilde{KO}_{-1}(C\pr(S^3, \tau^{3,1}))$. 
Furthermore, it must be a generator, because $c_{-1}[x_{-1}] = [y_1]$ is a generator of $\widetilde{KU}_{1}(C\pr(S^3, \tau^{3,1}))$. 

We know that $\eta_{-1}[x_{-1}]$ is a generator for $\widetilde{KO}_{0}(C\pr(S^3, \tau^{3,1}))$ and the formula for $\eta_{-1}$ from Theorem~5.1 of \cite{Boer2020} is
$$\eta_{-1}[u] = \sm{a}{ib}{-ib}{-a}$$
where 
$$a = \tfrac{1}{2}(u + u^*) \quad \text{and} \quad b = \tfrac{1}{2}(u - u^*) \; .$$
Using this we get the formula
$$ \left( \begin{smallmatrix}
			x & y + iw & -z & 0 \\ y - iw & -x & 0 & -z \\ -z & 0 & -x & -y - iw \\ 0 & -z & -y + iw & x 
		\end{smallmatrix} \right) \; $$
for a generator of $\widetilde{KO}_{0}(C\pr(S^3, \tau^{3,1})) = \Z$.
For convenience, we replaced the $-z$ with $z$ everywhere in the formula for $x_0$, which gives us the opposite generator.

We also know that $\eta_0[x_0]$ is a generator for $\widetilde{KO}_{1}(C\pr(S^3, \tau^{3,1}))$ and the formula for $\eta_{0}$ is
$\eta_0[u] = u \cdot I_4^{(0)}$ where
$I_4^{(0)} = \diag(1,-1,1,-1)$. Thus we take $x_1 = x_0 \cdot I_4^{(0)}$ to get the formula in the table. 
We can double-check that $x_1^\tau = x_1^*$ to verify that $[x_1]$ is an element of $\widetilde{KO}_{1}(C\pr(S^3, \tau^{3,1}))$.

Finally, we obtain the formula for $x_3$ from $y_1$, using the fact that $r_3$ is an isomorphism in this case. The formula for $r_3$ from Theorem~3.2 of \cite{Boer2020} is
$$r_3[u]  = \left[ \sm{a}{ib}{-ib}{a}  \right]$$
where 
$$a = \tfrac{1}{2} (u + u\stau) \quad \text{and} \quad b = \tfrac{1}{2}(u - u\stau) \; .$$
When we do this, we obtain
$$x_3' =  i \left( \begin{smallmatrix}
		z & w & x & y \\ -w & z & y & -x \\ -x & -y & z & w \\ -y & x & -w & z	
		\end{smallmatrix} \right)	$$ 
which is therefore a representative for 	$\widetilde{KO}_{3}(C\pr(S^3, \tau^{3,1}))$.
But note that we can get rid of the $i$ coefficient for $x_3'$ by letting $\theta(t)$ be a path of unit complex numbers from $1$ to $-i$ and taking
$\theta(t) \cdot x_3'$. The required $KO_3(-)$ relation $u^{\sharp \otimes \tau} = u$ is preserved along this path of unitaries from $x_3'$ to $x_3$, because $\sharp \otimes \tau$ is complex-linear.
\end{proof}

\begin{table}[h]
\caption{Unitaries for $\widetilde{K}\crr(C\pr(S^3, \tau^{2,2}))$.} \label{S3unitaryC}
$$\begin{array}{|c|c|c|}  \hline
 (x,y,z, w) \mapsto (x,y,-z, -w)  & \text{isomorphism class} & { \text{unitary representing a generator}}  \\ \hline \hline
 \widetilde{KU}_0 & 0 &  \\ \hline
\widetilde{KU}_1 & \Z &   y_1 = \sm{x + iw}{y + iz}{y - iz}{-x + iw}		       \\ \hline \hline
\widetilde{KO}_0 & 0 & 				\\ \hline 
\widetilde{KO}_1 &  \Z &  x_1 = \sm{x + iw}{-y - iz}{y - iz}{x - iw}					\\ \hline 
\widetilde{KO}_2 & \Z_2 &  x_2 = \left( \begin{smallmatrix}
			0 & -w + ix & 0 &  z - iy \\ -w -ix & 0 & z - iy & 0 \\ 0 & z+iy & 0 & w + ix \\ z + iy &0 & w - ix & 0
		\end{smallmatrix} \right)		\\ \hline 
\widetilde{KO}_3 & \Z_2 & 		x_3 = \left( \begin{smallmatrix}
			x + iw & -iz & 0 & -iy \\ -iz & x-iw & iy & 0 \\ 0 & iy & x+iw & -iz \\ -iy & 0 & -iz & x - iw 
		\end{smallmatrix} \right)		\\ \hline 
\widetilde{KO}_4 & 0 & 				 \\ \hline
\widetilde{KO}_5 & \Z & 		x_5 
		= \left( \begin{smallmatrix}
	x & -y & -w & z \\ y & x & z & w \\ w & -z & x & -y \\ -z & -w & y & x
		\end{smallmatrix} \right) 	\\ \hline 
\widetilde{KO}_6 & 0 & 				\\ \hline 
\widetilde{KO}_7 & 0 & 						\\ \hline 
\end{array}$$
\end{table}

\begin{prop}
Table \ref{S3unitaryC} shows unitary representatives of generators for $\widetilde{K}\crr(C\pr(S^3, \tau^{2,2}))$.
\end{prop}

\begin{proof}[Sketch of Proof]
The choice of $y_1$ is same as in previous examples, up to rearranging and sign changes of variables.

	
The involution is given by $(x,y,z,w)^\tau = (x,y,-z-,-w)$. We take $x_1 = y_1$ and check that $x_1^\tau = x_1^*$. 
We also check that 
$$\text{ $\ev_*([x_1]) = [\diag(1, 1) ]   = 0$ in $KO_1(\R)$. }$$
where $\ev$ is evaluation at the point $(1,0,0,0) \in S^3$.
This implies that  $[x_1] \in \ker (\ev_*)$ and hence that $[x_1] \in \widetilde{KO}_1(C\pr(S^3,\tau^{2,2}))$.
Now note that 
$x_1 = y_1 \cdot \diag(1, -1)$.
Thus
$$c_1([x_1]) = [x_1] = [y_1 \cdot \diag(1,-1)] = [y_1] \; .$$
Since $c_1$ is an isomorphism for this $C \sp *$-algebra, it follows that $[x_1]$ must be a generator of 
$\widetilde{KO}_1(C\pr(S^3,\tau^{2,2}))$, using the same reasoning as in the previous cases.

Now to obtain a formula for $x_2$, we use the fact that $\eta_1$ carries a generator of $\widetilde{KO}_1(C\pr(S^3, \tau^{2,2}))$ to the non-trivial element of
$\widetilde{KO}_2(C\pr(S^3, \tau^{2,2}))$. The formula for $\eta_1$ from Theorem~5.1 of \cite{Boer2020} is
$$\eta_1[u] = X \sm{0}{iu}{-iu^*}{0} X^*$$
where $X$ is a unitary such that conjugation by $X$ swaps the second and third rows and columns of a $2 \times 2$ matrix. Applying this to $x_1$ then gives the formula for $x_2$ in Table~\ref{S3unitaryC}.
			
To obtain $x_3$ we use the fact that $r_3[y_1]$ must represent the non-trivial element of $\widetilde{KO}_3(C\pr(S^3, \tau^{2,2}))$. By Theorem~1 of \cite{Boer2020}, we have
$$r_3[u] = \sm{a}{ib}{-ib}{a}$$
where 
$$a = \tfrac{1}{2}\left( u + u\stau \right) \quad \text{and} \quad b = \tfrac{1}{2}\left( u - u\stau \right) \; .$$	
This gives us $x_3$ as shown.

Finally, to obtain $x_5$ we also use the fact that $r_5[y_1]$ must represent a generator of $\widetilde{KO}_3(C\pr(S^3, \tau^{2,2}))$. This formula for $r_5$ from \cite{Boer2020} is
$$r_5[u] = \sm{a}{ib}{-ib}{a}$$
where 
$$a = \tfrac{1}{2}\left( u + u^{*(\sharp \otimes \tau)} \right) \quad \text{and} \quad b = \tfrac{1}{2}\left( u - u^{*(\sharp \otimes \tau)} \right) \; .$$	

\end{proof}

\vspace{1cm}
\begin{prop}
Table \ref{S3unitaryD} shows unitary representatives of generators for $\widetilde{K}\crr(C\pr(S^3, \tau^{1,3}))$.
\end{prop}

\begin{table}[h]
\caption{Unitaries for $\widetilde{K}\crr(C\pr(S^3, \tau^{1,3}))$.} \label{S3unitaryD}
$$\begin{array}{|c|c|c|}  \hline
 (x,y,z, w) \mapsto (-x,-y, -z, w)  & \text{isomorphism class} & { \text{unitary representing a generator}}  \\ \hline \hline
 \widetilde{KU}_0 & 0 &  \\ \hline
 \widetilde{KU}_1 & \Z &   y_1 = \sm{x + iw}{y + iz}{y - iz}{-x + iw}       \\ \hline \hline
\widetilde{KO}_0 & 0 & 				\\ \hline 
\widetilde{KO}_1 &  0 &  				\\ \hline 
\widetilde{KO}_2& 0 &  				\\ \hline
\widetilde{KO}_3 & \Z & 	x_3 = \sm{x + iw}{y + iz}{y - iz}{-x + iw} 			 \\ \hline 
\widetilde{KO}_4 & \Z_2 & x_4 = 
		\left( \begin{smallmatrix}
			x  & y+ iz & w & 0 \\  y - iz & -x & 0 & w \\  w & 0 & -x & - y - iz \\ 0 & w & -y + iz & x					
	\end{smallmatrix} \right)  			 \\ \hline
\widetilde{KO}_5& \Z_2 & x_5 = 	
		\left( \begin{smallmatrix}
			x  & y+ iz & w & 0 \\  y - iz & -x & 0 & w \\  -w & 0 & x & y+ iz \\ 0 & -w & y - iz & -x				
	\end{smallmatrix} \right)  			 \\ \hline
\widetilde{KO}_6 & 0 & 		
									\\ \hline 
\widetilde{KO}_7 & \Z & x_7 =
		 i  \left( \begin{smallmatrix}
			w & z & x & y \\ -z & w & y & -x \\ -x & -y & w & z \\ -y & x & -z & w				
	\end{smallmatrix} \right) 		\\ \hline 
\end{array}$$
\end{table}

\vspace{1cm}

\begin{proof}
We use the same $y_1$ in earlier examples.

Check that $x_3$ satisfies $x_3^{\sharp \otimes \tau} = x_3$. Then $[x_3]$ must represent the generator of $\widetilde{K}\crr(C\pr(S^3, \tau^{1,3}))$
since $c_3$ is an isomorphism.

Using the fact that $\eta_3$ is surjective, we find a formula for $x_4$ from the expression $[x_4] = \eta_3[x_3]$.
The formula for $\eta_3$ is
$$\eta_3[u] = \sm{a}{ib}{ib}{-a}$$
where 
$$a = \tfrac{1}{2}\left( u + u^* \right) \quad \text{and} \quad b = \tfrac{1}{2}\left( u - u^* \right) \; .$$	

We find $x_5$ using the fact that $r_5$ is surjective and the expression $r_5([y_1]) = [x_5]$. The formula for $r_5$ is
$$r_5[u] = \sm{a}{ib}{-ib}{a}$$
where 
$$a = \tfrac{1}{2}\left( u + u^{*(\sharp \otimes \tau)} \right) \quad \text{and} \quad b = \tfrac{1}{2}\left( u - u^{*(\sharp \otimes \tau)} \right) \; .$$
We apply this to $y_1$ to find $x_5$.	

Finally, we also find $x_7$ using $r_7$ which is an isomorphism. The formula for $r_7$ is
$$r_7[u] = \sm{a}{ib}{-ib}{a}$$
where 
$$a = \tfrac{1}{2}\left( u + u^\tau \right) \quad \text{and} \quad b = \tfrac{1}{2}\left( u - u^\tau \right) \; .$$
\end{proof}

\vspace{2cm}


\section{$K$-Theory Generators in Higher Dimensions} \label{section-highdim}

Now, we consider the general case of the spheres $(S^d, \tau^{a,b})$ with $a \geq 1$ and no restriction on $d$. We will describe general algorithms for finding the key generator
$$g_{a,b} \in KO_{b-a+1}(C\pr(S^d, \tau^{a,b})) = \Z$$
in each case.
Prior to the general case, we consider two special cases of involutions of the form $\tau^{a,0}$ and $\tau^{1, b}$ in Subsections~\ref{sub1} and \ref{sub2}. The most general situation $\tau^{a,b}$ is addressed in Subsection~\ref{sub3}.

For these general recipes for producing generators of $K\crr(C\pr(S^d, \tau^{a,b}))$ we introduce the so-called $k$-Clifford generators, 
as used in \cite{SBS} and in \cite{boersema-spheres1}.

\begin{defn}
Let $k$ be a positive odd integer and let  $n = 2^{(k-1)/2}$. A list of self-adjoint matrices $a_1, \dots, a_k \in M_n(\C)$ is called a list of ``complex $k$-Clifford generators" if they satisfy 
$$a_i a_j + a_j a_i = 2 \delta_{i,j} \; $$
for all $i,j$. These elements form irreducible representations of the real Clifford algebras $\mathcal{C}_k\pr$.
\end{defn}

\begin{const}  \label{construction0}
For our purposes, and following Section 2 of \cite{SBS}, we declare 
the following to be the ``standard" list of complex self-adjoint complex $k$-Clifford generators in $M_n(\C)$ constructed inductively:
\begin{itemize} 
\item For $k = 1$, we have $\Gamma_{1,1} = 1$.
\item For $k = 3$, we have 
$$\Gamma_{1,3} = \begin{bmatrix} 0 & 1 \\ 1 & 0 \end{bmatrix} \; , \quad
\Gamma_{2,3} =  \begin{bmatrix} 0 & i \\ -i & 0 \end{bmatrix} \; , \quad
\Gamma_{3,3} =  \begin{bmatrix} 1 & 0 \\ 0 & -1 \end{bmatrix} \; .$$
\item For $k \geq 3$, let $\Gamma_{1,k}, \dots, \Gamma_{k,k}$ be the standard list of complex self-adjoint $k$-Clifford generators in $M_n(\C)$, and then define $\Gamma_{1,k+2}, \dots, \Gamma_{k+2,k+2} \in M_{2n}(\C)$ by
\begin{align*}
\Gamma_{i, k+2} &= \begin{bmatrix} 0 & \Gamma_{i, k} \\ \Gamma_{i,k} & 0 \end{bmatrix} && i = 1, \dots, k, \\
\Gamma_{k+1,k+2} &=  \begin{bmatrix} 0 & i I_{n} \\ -i I_{n} & 0 \end{bmatrix} \\
\Gamma_{k+2, k+2} &=  \begin{bmatrix} I_{n} & 0 \\ 0 & -I_{n}  \end{bmatrix} \\
\end{align*}
\end{itemize}
\end{const}

\begin{lemma}
For positive odd integers $k$, the standard list of complex self-adjoint $k$-Clifford generators satisfies
$$
(\Gamma_{i,k})\T = (-1)^{i+1} \Gamma_{i,k} 
\quad \text{and} \quad
(\Gamma_{i,k})\sT = 
\begin{cases} 
-\Gamma_{i,k} & i \leq 3 \\
(-1)^{i+1} \Gamma_{i,k} & \rm{otherwise.}
\end{cases} $$
\end{lemma}

\begin{proof}
Check this directly for $k = 3$ and then proceed by induction.
\end{proof}

The following lemma allows us to modify any set of $k$-Clifford generators to toggle the behavior of the transpose operator or the sharp operator, in groups of four. As we will see the significance of this is closely related to the period-8 periodicity of real $K$-theory.

\begin{lemma} \label{toggle}
Let $\Sigma = \{\sigma_1, \sigma_2, \dots, \sigma_n\}$ be any set of $k$-Clifford generators, and let 
$S = \{\sigma_{i_1}, \dots, \sigma_{i_k}\}$ be an (ordered) subset of these generators where $|S| = k \equiv 0 \pmod 4$.
Define $\widetilde{\sigma}_{i_1}, \dots, \widetilde{\sigma}_{i_k}$ by
$$
	\widetilde{\sigma}_{i_j} = i \sigma_{i_1} \cdots \widehat{\sigma}_{i_j} \cdots \sigma_{i_k} \; .
$$
Then the set $\widetilde \Sigma$ obtained from $\Sigma$ by replacing each $\sigma_{i_j}$ with $\widetilde{\sigma}_{i_j}$ is also a set of $k$-Clifford generators.
Furthermore, 
\begin{enumerate}
\item If $\sigma_{i_j}\T = \sigma_{i_j}$ for all $j$, then $\widetilde{\sigma}_{i_j}\T = -\widetilde{\sigma}_{i_j}$ for all $j$.
\item If $\sigma_{i_j}\T = -\sigma_{i_j}$ for all $j$, then $\widetilde{\sigma}_{i_j}\T = \widetilde{\sigma}_{i_j}$ for all $j$.
\item If $\sigma_{i_j}\sT = \sigma_{i_j}$ for all $j$, then $\widetilde{\sigma}_{i_j}\sT = -\widetilde{\sigma}_{i_j}$ for all $j$.
\item If $\sigma_{i_j}\sT = -\sigma_{i_j}$ for all $j$, then $\widetilde{\sigma}_{i_j}\sT = \widetilde{\sigma}_{i_j}$ for all $j$.
\end{enumerate}
	 \end{lemma}

\begin{proof}
This follows immediately from Lemmas~4.4 and 4.5 from \cite{boersema-spheres1}.
\end{proof}

\vspace{1cm}


\subsection{{\bf Special Case I: $C\pr(S^{d}) = C\pr(S^{d}, \tau^{a,0})$ }} \label{sub1}

~

In this section, we consider spheres with the trivial involution $\id = \tau^{a,0}$ on $S^d$, where $a = d+1$.
Recall that
$ \widetilde{KO}_*(C\pr(S^d))$ is the free $\mathcal{CR}$-module with generator
$g_{a,0} = g_{d+1,0} \in KO_{-d}(C\pr(S^d)) \cong \Z$. We will identity unitary representations of this generator for all positive integers $a$.

\vspace{.5cm}

\begin{const} ~ \label{construction1}

For $k$ odd, this construction shows how to produce particular unitaries $Q_{k-1}$ and $U_k$, which we subsequently prove do represent the generators of the appropriate $KO$-groups. 
In each case, start with the standard complex Clifford generators $\Gamma_{1,k}, \dots, \Gamma_{k,k} \in M_n(\C)$ where $n = 2^{(k-1)/2}$.
The claims in the description of the construction will be justified by 
Theorem~\ref{main-thmI} below.

\begin{enumerate}

\item[(1)] Let $k \equiv -1 \pmod 8$.
\begin{itemize}
\item Let $$  S=  \{j \in  \{1, \dots, k\} \mid \text{$j$ is odd} \} \; .$$
\item Define $\Up_{1,k}, \dots, \Up_{k,k} \in M_n(\C)$ by
$$\Up_{i,k} =
	\begin{cases} 
			i  \prod_{ j \in S \backslash \{i\}}  \Gamma_{j,k} & i \in S \\ \hspace{.5cm} \Gamma_{i,k} & i \notin S 
	\end{cases} $$
\item Then $\Up\T_{i,k} =  -\Up_{i,k} \; $ for all $i$.
\item Let $Q\pr_{k-1}(x) = \sum_{i = 1}^k x_i \Up_{i,k}$ and $U\pr_k(x) =  \sum_{i = 1}^k  x_i i \Up_{i,k} + x_{k+1}  I_n$.
\item Then $[Q\pr_{k-1}] \in \widetilde{KO}_{2}(C\pr(S^{k-1}))$ and $[U\pr_k] \in \widetilde{KO}_{1}(C\pr(S^{k}))$.
\end{itemize}

\vspace{.5cm}

\item[(2)] Let $k \equiv 1 \pmod 8$.
\begin{itemize}
\item Let $$  S=  \{j \in  \{1, \dots, k\} \mid \text{$j$ is even} \}.  $$
\item Define $\Up_{1,k}, \dots, \Up_{k,k} \in M_n(\C)$ by
$$\Up_{i,k} =
	\begin{cases} 
			i  \prod_{ j \in S \backslash \{i\}}   \Gamma_{j,k} & i \in S \\ \hspace{.5cm} \Gamma_{i,k} & i \notin S 
	\end{cases} $$
\item Then $\Up\T_{i,k} =  \Up_{i,k} \; $ for all $i$.
\item Let $Q\pr_{k-1}(x) = \sum_{i = 1}^k x_i \Up_{i,k}$ and $U\pr_k(x) = \sum_{i = 1}^k x_i \Up_{i,k} + x_{k+1} i I_n$.
\item Then $[Q\pr_{k-1}] \in \widetilde{KO}_{0}(C\pr(S^{k-1}))$ and $[U\pr_k] \in \widetilde{KO}_{-1}(C\pr(S^{k}))$.
\end{itemize}

\vspace{.5cm}

\item[(3)] Let $k \equiv 3 \pmod 8$.
\begin{itemize}
\item Let $$  S=  \{j \in  \{1, \dots, k\} \mid \text{$j$ is odd and $j \geq 5$} \} . $$
\item Define $\Up_{1,k}, \dots, \Up_{k,k} \in M_n(\C)$ by
$$\Up_{i,k} =
	\begin{cases} 
			 i \prod_{ j \in S \backslash \{i\}} \Gamma_{j,k} & i \in S \\ \hspace{.5cm} \Gamma_{i,k} & i \notin S 
	\end{cases} $$
\item Then $\Up\sT_{i,k} =  - \Up_{i,k} \; $ for all $i$.
\item Let $Q\pr_{k-1}(x) = \sum_{i = 1}^k x_i \Up_{i,k}$ and $U\pr_k(x) =  \sum_{i = 1}^k  x_i i \Up_{i,k} + x_{k+1}  I_n$.
\item Then $[Q\pr_{k-1}] \in \widetilde{KO}_{6}(C\pr(S^{k-1}))$ and $[U\pr_k] \in \widetilde{KO}_{5}(C\pr(S^{k}))$.
\end{itemize}

\vspace{.5cm}

\item[(4)] Let $k \equiv 5 \pmod 8$.
\begin{itemize}
\item Let $$  S=  \{j \in  \{1, \dots, k\} \mid \text{$j$ is even or $j \leq 3$} \} . $$
\item Define $\Up_{1,k}, \dots, \Up_{k,k} \in M_n(\C)$ by
$$\Up_{i,k} =
	\begin{cases} 
			i  \prod_{ j \in S \backslash \{i\}} \Gamma_{j,k} &i   \in S \\ \hspace{.5cm} \Gamma_{i,k} & i \notin S 
	\end{cases} $$
\item Then $\Up\sT_{i,k} =  \Up_{i,k} \; $ for all $i$.
\item Let $Q\pr_{k-1}(x) = \sum_{i = 1}^k x_i \Up_{i,k}$ and $U\pr_k(x) = \sum_{i = 1}^k x_i \Up_{i,k} + x_{k+1} i I_n$.
\item Then $[Q\pr_{k-1}] \in \widetilde{KO}_{4}(C\pr(S^{k-1}))$ and $[U\pr_k] \in \widetilde{KO}_{3}(C\pr(S^{k}))$.
\end{itemize}
\vspace{.5cm}
\end{enumerate}
\end{const}



\begin{thm} \label{main-thmI}
In each case, $[Q\pr_{k-1}]$ and $[U\pr_k]$ from Parts (1)-(4) of Construction~\ref{construction1} represent generators of 
$\widetilde{KO}_{-(k-1)}(C\pr(S^{k-1})) = \Z$ and $\widetilde{KO}_{-k}(C\pr(S^{k})) = \Z$ respectively, for odd $k$. 
In other words,
$$g_{k,0} = [Q\pr_{k-1}] \qquad \text{and} \qquad g_{k+1,0} = [U\pr_k] \; .$$
\end{thm}

\begin{proof}
Note that in each case of Construction~\ref{construction1} we have $|S| \equiv 0 \pmod 4$, so Lemma~\ref{toggle} can be used in the way indicated to create the new set of Clifford generators.
We also note that in each case the set $S$ consists of exactly the indices for the set of Clifford generators $\Gamma_{i,k}$ that do not satisfy the desired symmetry. Then the new adjusted set of of Clifford generators $\Up_{i,k}$ do all satisfy the symmetry indicated in the third bullet point. 

It is easy to check that $Q\pr_{k-1}$ is a self-adjoint unitary and that $U\pr_k$ is a unitary in each case, from the fact that
$\Up_{i,k}$ are self-adjoint unitaries. We check that these unitaries satisfy the correct symmetry in each case.
In the computations below, we write $\tau = \tau^{k,0}$ or $\tau = \tau^{k+1,0}$ as appropriate in the context in each case.

In the case $k \equiv -1 \pmod 8$, the Clifford generators $\Up_{1, k}, \dots, \Up_{k,k}$ are antisymmetric, hence
\begin{align*}
	(Q\pr_{k-1}) \Ttau(x)   
		&= \left( \sum_{i = 1}^k x_i \Up_{i,k} \right) \Ttau \\
		&= \sum_{i = 1}^{k} (x_i)^\tau \Up_{i,k}\T\\
		&= \sum_{i = 1}^{k} x_i  (- \Up_{i,k}) \\
		&= -Q\pr_{k-1} (x) \;  \\
\text{and} \quad 
	(U\pr_k)\Ttau(x)   
		&= \left( \,  \sum_{i = 1}^{k} x_i i \, \Up_{i,k}  + x_{k+1} I_n \right) \Ttau \\
		&=  \sum_{i = 1}^{k} x_i i \, (-\Up_{i,k})  + x_{k+1} I_n   \\
		&=   - \sum_{i = 1}^{k} x_i  i \, \Up_{i,k} +  x_{k+1}  I_n   \\
		&=  (U\pr_k)^*(x) \; .
\end{align*}
Therefore $[Q\pr_{k-1}] \in KO_2(C\pr(S^{k-1}))$ and $[U\pr_{k}] \in KO_1(C\pr(S^{k}))$.

In the case that $k \equiv 1 \pmod 8$, the Clifford generators $\Up_{1, k}, \dots, \Up_{k,k}$ are symmetric, so we have
\begin{align*}
	(Q\pr_{k-1})\Ttau(x)   
		&= \left( \sum_{i = 1}^k x_i \Up_{i,k} \right) \Ttau \\
		&= \sum_{i = 1}^{k} x_i(  \Up_{i,k})\T \\
		&= \sum_{i = 1}^{k} x_i  \Up_{i,k} \\
		&= Q\pr_{k-1} (x) \;  \\
\text{and} \quad 
	(U\pr_k)\Ttau(x)   
		&= \left( \,  \sum_{i = 1}^{k} x_i \Up_{i,k}  + x_{k+1} i \,  I_n \right) \Ttau \\
		&=   \sum_{i = 1}^{k} x_i  (\Up_{i,k})\T  + x_{k+1} i \, \left(I_n \right)\T  \\
		&=  \sum_{i = 1}^{k} x_i  \Up_{i,k}  + x_{k+1} i \, I_n   \\
		&= U\pr_k(x) \; .
\end{align*}
It follows that $[Q\pr_{k-1}] \in \widetilde{KO}_0(C\pr(S^{k-1}))$ and $[U\pr_{k}] \in \widetilde{KO}_{-1}(C\pr(S^{k}))$.

In the third case, $k \equiv 3 \pmod 8$, the elements $\Up_{1, k}, \dots, \Up_{k,k}$ satisfy $\Up_{i,k}\sT = - \Up_{i,k}$. Thus 
similar calculations to above show that $(Q\pr_{k-1}) \sTtau = -Q\pr_{k-1}$  and that $(U\pr_k)\sTtau =  (U_k\pr)^*$ and we obtain 
$[Q\pr_{k-1}] \in  \widetilde{KO}_6(C\pr(S^{k-1}))$ and $[U\pr_{k}] \in \widetilde{KO}_5(C\pr(S^{k}))$.


Finally, if $k \equiv 5 \pmod 8$ we have $\Up_{i,k}\sTtau = \Up_{i,k}$ for all $i$ so
$(Q\pr_{k-1}) \sTtau = Q\pr_{k-1}$ and $(U\pr_k)\sTtau = U\pr_k$ which implies that $[Q\pr_{k-1}] \in  \widetilde{KO}_4(C\pr(S^{k}))$ and $[U\pr_{k}] \in  \widetilde{KO}_{3}(C\pr(S^{k}))$.


Therefore, we have shown that for all $k$,
$$[Q\pr_{k-1}] \in  \widetilde{KO}_{-(k-1)}(C\pr(S^{k-1})) \quad \text{and} \quad 
	[U\pr_k] \in  \widetilde{KO}_{-k}(C\pr(S^{k})) \; .$$ 
It remains to show that these elements are in fact generators of the (reduced) $KO$-group.

The complexification map 
$$c_{-d} \colon \widetilde{KO}_{-d}(A) \rightarrow \widetilde{KU}_{-d}(A)$$ is defined by forgetting the extra symmetry that a self-adjoint unitary representing real $K$-theory must satisfy by Theorem~3.1 of \cite{Boer2020}. In this case at hand with $A = C\pr(S^{d})$, we know that $c_{-d}$ is an isomorphism, because of the structure of the free $\mathcal{CR}$-module ---  namely that
$\widetilde{K}_*\crr(C\pr(S^{d}))$ is a free $\mathcal{CR}$-module with generator in $\widetilde{KO}_{-d}(C\pr(S^{d})) = \Z$.

Thus 
$$c_{-(k-1)}\left( [Q\pr_{k-1}] \right) = [Q_{k-1}] \in \widetilde{KU}_{-(k-1)}(C\pr(S^{k-1})) = \Z $$
where 
$$Q_{k-1}(x) = \sum_{i = 1}^k x_i \Up_{i,k}$$
Now, from Proposition~1 of \cite{SBS} we know that $[Q_{k-1}]$ in each of the four cases represents a generator of 
$\widetilde{KU}_{-(k-1)}(C\pr(S^{k-1})) = \Z$, because each is defined in terms of a set of Clifford generators.
Therefore, it follows that $[Q\pr_{k-1}]$ must be a generator of $\widetilde{KO}_{-(k-1)}(C\pr(S^{k-1})) = \Z$.

The same argument applies for the $[U\pr_k]$ with the following observation.
For $k \equiv 1 \pmod 8$ and $k \equiv 5 \pmod 8$ we have
$$U_k(x) =  \sum_{i = 1}^k  x_i \Up_{i,k} + i x_{k+1}  I_n$$
which matches the definition of $U_k$ in Proposition~1 of \cite{SBS}.
However, for 
$k \equiv -1 \pmod 8$ and $k \equiv 3 \pmod 8$ we have
$$U_k(x) =  \sum_{i = 1}^k  i x_i \Up_{i,k} + x_{k+1}  I_n \; $$
which differs to that of \cite{SBS}. However, recall that for complex $K$-theory, we have $[u] = [iu] \in KU_1(A)$ for any unitary $u \in M_n(A_\C)$.
Furthermore, we also have $[u^*] = -[u] \in KU_1(A)$. It follows from this that in all cases $[U_k]$ is a generator of
$\widetilde{KU}_{-k}(C\pr(S^{k})) = \Z$ and therefore that
$[U\pr_k]$ is a generator of $\widetilde{KO}_{-k}(C\pr(S^{k})) = \Z$.

Therefore, for all odd $k$,
$$g_{k,0} = [Q\pr_{k-1}] \quad \text{and} \quad g_{k+1,0} = [U_k\pr] \; .$$
\end{proof}

\newpage


\subsection{{\bf Special Case II:  $C\pr(S^{d}, \tau^{1,b})$}} \label{sub2}

~

In this section, we consider spheres for which the involution $\tau^{1,b}$ on $S^d$ (where $b = d$) is the the antipodal action in every coordinate except one.
Here $\widetilde{KO}_*(C\pr( S^{d}, \tau^{1, b}))$ is a free $\mathcal{CR}$-module with generator $g_{1,b} \in  \widetilde{KO}_{b}(C(S^{d}, \tau^{1,b})) = \Z$.
Hence we are looking for unitaries that represent the key generators
$g_{1,b} \in  \widetilde{KO}_{b}(C\pr(S^{d}, \tau^{1,b})) = \Z$.

Establishing some notation, for an integer $n$, define 
$$\delta_i^n  = \begin{cases}
					1 & i = n \\ -1 & i \neq n  \; .
			\end{cases} $$
\vspace{1cm} ~ \\ ~

\begin{const} ~ \label{construction2}

For $k$ odd, this construction shows how to produce particular unitaries $Q_{k-1}$ and $U_k$, which ultimately we show to represent the generators of the $KO$-groups indicated, as indicated in Theorem~\ref{main-thm2} below.


\begin{enumerate}

\item[(1)] Let $k \equiv 1 \pmod 8$.
\begin{itemize}
\item Let $$  S=  \{j \in  \{1, \dots, k\} \mid \text{ $j$ is odd, $j \neq 1$} \} \; .$$
\item Define $\Up_{1,k}, \dots, \Up_{k,k} \in M_n(\C)$ by
$$\Up_{i,k} =
	\begin{cases} 
			i  \prod_{ j \in S \backslash \{i\}}  \Gamma_{j,k} & i \in S \\ \hspace{.5cm} \Gamma_{i,k} & i \notin S  \; .
	\end{cases} $$
\item Then 
$\Up\T_{i,k} =  \delta_i^1 \Up_{i,k} \; $ for all $i$.
\item Arrange the variables in $\R^{k+1}$ so that the involution $\tau^{1,k}$ on $S^{k} \subset \R^{k+1}$ is given by $x_i \mapsto \delta_i x_i$.
\item Let $Q\pr_{k-1}(x) = \sum_{i = 1}^k x_i \Up_{i,k}$ and $U\pr_k(x) =  \sum_{i = 1}^k   x_i  \Up_{i,k} + x_{k+1} \; i  I_n$.
\item Then 
		$[Q\pr_{k-1}] \in \widetilde{KO}_0( C\pr(S^{k-1}, \tau^{1, k-1}))$ 
		and $[U\pr_k] \in \widetilde{KO}_{1}(C\pr(S^{k}, \tau^{1,k}))$.
\end{itemize}

\vspace{.5cm}

\item[(2)] Let $k \equiv 3 \pmod 8$.
\begin{itemize}
\item Let $$  S=  \{j \in  \{1, \dots, k\} \mid \text{ $j$ is even, $j \neq 2$} \} \; .$$

\item Define $\Up_{1,k}, \dots, \Up_{k,k} \in M_n(\C)$ by
$$\Up_{i,k} =
	\begin{cases} 
			i  \prod_{ j \in S \backslash \{i\}}  \Gamma_{j,k} & i \in S \\ \hspace{.5cm} \Gamma_{i,k} & i \notin S  \; .
	\end{cases} $$
\item Then $\Up\T_{i,k} = - \delta_i^2 \Up_{i,k}$ for all $i$.
\item Arrange the variables in $\R^{k+1}$ so that the involution $\tau^{1,k}$ on $S^{k} \subset \R^{k+1}$ is given by 
	$x_i \mapsto \delta_i^2 x_i$.
\item Let $Q\pr_{k-1}(x) = \sum_{i = 1}^k x_i \Up_{i,k}$ and $U\pr_k(x) =  \sum_{i = 1}^k  x_i  \Up_{i,k} + x_{k+1} i  I_n$.
\item Then 
		$[Q\pr_{k-1}] \in \widetilde{KO}_2( C\pr(S^{k-1}, \tau^{1, k-1}))$ 
		and $[U\pr_k] \in \widetilde{KO}_{3}(C\pr(S^{k}, \tau^{1,k}))$.
\end{itemize}
\vspace{.5cm}

\item[(3)] Let $k \equiv 5 \pmod 8$.
\begin{itemize}
\item Let $$  S=  \{j \in  \{1, \dots, k\} \mid \text{$j$ is odd, $j \geq 7$} \} . $$
\item Define $\Up_{1,k}, \dots, \Up_{k,k} \in M_n(\C)$ by
$$\Up_{i,k} =
	\begin{cases} 
			i  \prod_{ j \in S \backslash \{i\}}  \Gamma_{j,k} & i \in S \\ \hspace{.5cm} \Gamma_{i,k} & i \notin S  \; .
	\end{cases} $$
\item Arrange the variables in $\R^{k+1}$ so that the involution $\tau^{1,k}$ on $S^{k} \subset \R^{k+1}$ is given by $x_i \mapsto \delta_i^5 x_i$ 
\item Then $\Up\sT_{i,k} =   \delta_i^5 \Up_{i,k}$ for all $i$.
\item Let $Q\pr_{k-1}(x) = \sum_{i = 1}^k x_i \Up_{i,k}$ and $U\pr_k(x) =  \sum_{i = 1}^k  x_i  \Up_{i,k} + x_{k+1} \; i I_n$.
\item Then 
		$[Q\pr_{k-1}] \in \widetilde{KO}_4( C\pr(S^{k-1}, \tau^{1, k-1}))$ 
		and $[U\pr_k] \in \widetilde{KO}_{5}(C\pr(S^{k}, \tau^{1,k}))$.
\end{itemize}
\vspace{.5cm}

\item[(4)] Let $k \equiv 7 \pmod 8$.
\begin{itemize}
\item Let $$  S=  \{j \in  \{1, \dots, k\} \mid \text{$j$ is even or $j =3$} \} . $$
\item Define $\Up_{1,k}, \dots, \Up_{k,k} \in M_n(\C)$ by
$$\Up_{i,k} =
	\begin{cases} 
			i  \prod_{ j \in S \backslash \{i\}}  \Gamma_{j,k} & i \in S \\ \hspace{.5cm} \Gamma_{i,k} & i \notin S  \; .
	\end{cases} $$
\item Then $\Up\sT_{i,k} = -  \delta_i^2 \Up_{i,k}$ for all $i$.
\item Arrange the variables in $\R^{k+1}$ so that the involution $\tau^{1,k}$ on $S^{k} \subset \R^{k+1}$ is given by 
	$x_i \mapsto \delta_i^2 x_i$ 
\item Let $Q\pr_{k-1}(x) = \sum_{i = 1}^k x_i \Up_{i,k}$ and $U\pr_k(x) =  \sum_{i = 1}^k  x_i  \Up_{i,k} + x_{k+1} \; i I_n$.
\item Then $[Q\pr_{k-1}] \in \widetilde{KO}_{6}(C\pr(S^{k-1}, \tau^{1,{k-1}}))$ and $[U\pr_k] \in \widetilde{KO}_{7}(C\pr(S^{k}, \tau^{1,k}))$.
\end{itemize}
\vspace{.5cm}
\end{enumerate}

\end{const}

\begin{thm} \label{main-thm2}
The elements $[Q\pr_{k-1}]$ and $[U\pr_k]$ from Parts (1)-(4) of Construction~\ref{construction2} represent generators of 
$\widetilde{KO}_{k-1}(C\pr(S^{k-1}, \tau^{1,k-1})) = \Z$ and $\widetilde{KO}_{k}(C\pr(S^{k}, \tau^{1,k})) = \Z$ respectively, for odd $k$. 
In other words,
$$g_{1,k-1} = [Q\pr_{k-1}] \qquad \text{and} \qquad g_{1,k} = [U\pr_k] \; .$$
\end{thm}



\begin{proof}
Note that (as in the previous construction) in each case $|S| \equiv 0 \pmod 4$, so Lemma~\ref{toggle} justifies the modification of the Clifford generators.
We also note that in each case the set $S$ is {\it almost} exactly the indices for the set of Clifford generators $\Gamma_{i,k}$ that do not satisfy the desired symmetry (in fact, exactly so except for one coordinate). Then the new adjusted set of Clifford generators $\Up_{i,k}$ do all satisfy the desired symmetry except for one coordinate. This accounts for the formula in the third bullet point, which is modified by the
$\delta_i^n$ factor.

As before it is easily checked that $Q\pr_{k-1}$ is a self-adjoint unitary and that $U\pr_k$ is a unitary. Below we verify that in each case these unitaries satisfy the correct symmetry. Write $\tau = \tau^{k-1, 1}$ or $\tau = \tau^{k,1}$ as appropriate in the computations below.

In the case $k \equiv 1 \pmod 8$, the Clifford generators satisfy $\Up_{i,k} = \delta_i^1 \Up_{i,k}$ and the involution satisfies
$(x_i)^\tau = \delta_i^1 x_i$.
So we have
\begin{align*}
	(Q\pr_{k-1}) \Ttau(x)   
		&= \left( \sum_{i = 1}^k x_i \Up_{i,k} \right) \Ttau \\
		&= \sum_{i = 1}^{k} (x_i)^\tau \Up_{i,k}\T\\
		&= \sum_{i = 1}^{k} (\delta^1_i x_i )( \delta^1_i \Up_{i,k}) \\
		&= \sum_{i = 1}^{k} x_i   \Up_{i,k} \\
		&= Q\pr_{k-1} (x) \;  \\
\text{and} \quad 
	(U\pr_k)\Ttau(x)   
		&= \left( \,  \sum_{i = 1}^{k} x_i  \Up_{i,k}  + x_{k+1} \; i I_n \right) \Ttau \\
		&= \sum_{i = 1}^{k} (x_i)^\tau \Up_{i,k}\T +   (x_{k+1})^\tau \; i I_n\T \\		
		&= \sum_{i = 1}^{k} (\delta^1_i x_i) (\delta^1_i \Up_{i,k}) +   (\delta_{k+1}^1 x_{k+1}) \; i I_n \\	
		&=  \sum_{i = 1}^{k} x_i \Up_{i,k}  - x_{k+1} \; i I_n   \\
		&=  (U\pr_k)^*(x) \; .
\end{align*}
Therefore $[Q\pr_{k-1}] \in \widetilde{KO}_{0}( C\pr(S^{k-1}, \tau^{1, k-1}))$ 
		and $[U\pr_k] \in \widetilde{KO}_{1}(C\pr(S^{k}, \tau^{1,k}))$.

Now, assume that $k \equiv 3 \pmod 8$. The Clifford generators satisfy $\Up_{i,k} = -\delta_i^2 \Up_{i,k}$ and the involution satisfies
$(x_i)^\tau = \delta_i^2 x_i$. Hence,
\begin{align*}
	(Q\pr_{k-1}) \Ttau(x)   
		&= \left( \sum_{i = 1}^k x_i \Up_{i,k} \right) \Ttau \\
		&= \sum_{i = 1}^{k} (x_i)^\tau \Up_{i,k}\T\\
		&= \sum_{i = 1}^{k} (\delta_i^2 x_i )( -\delta_i^2 \Up_{i,k}) \\
		&= -\sum_{i = 1}^{k} x_i   \Up_{i,k} \\
		&= -Q\pr_{k-1} (x) \;  \\
\text{and} \quad 
	(U\pr_k)\Ttau(x)   
		&= \left( \,  \sum_{i = 1}^{k} x_i  \Up_{i,k}  + x_{k+1} i I_n \right) \Ttau \\
		&= \sum_{i = 1}^{k} ( x_i)^\tau ( \Up_{i,k}\T) +   (x_{k+1})^\tau \; i I_n\T \\	
		&= \sum_{i = 1}^{k} (\delta_i^2 x_i) (-\delta_i^2 \Up_{i,k}) +   (\delta_{k+1}^2 x_{k+1}) \; i I_n \\	
		&=  - \sum_{i = 1}^{k} x_i \Up_{i,k}  - x_{k+1} i  I_n   \\
		&=  - (U\pr_k)(x) \; .
\end{align*}
Therefore $[Q\pr_{k-1}] \in \widetilde{KO}_{2}( C\pr(S^{k-1}, \tau^{1, k-1}))$ 
		and $[U\pr_k] \in \widetilde{KO}_{3}(C\pr(S^{k}, \tau^{1,k}))$.
	(Where here we are using the {\it second} characterization given in the {\it second line} for $KO_3(A)$ in Table~\ref{unitaryTable}).
		
In the third case, for $k \equiv 5 \pmod 8$ we have $\Up_{i,k}\sT = \delta_i^5 \Up_{i,k}$ and $(x_i)^\tau = \delta_i^5 x_i$. Similar calculations to the above show that $(Q\pr_{k-1}) \sTtau = Q\pr_{k-1}$ and $(U\pr_{k-1}) \sTtau = (U\pr_{k-1})^*$ showing that the represent the appropriate real $K$-theory classes.


Finally,  in the fourth case for $k \equiv 7 \pmod 8$ we have $\Up_{i,k}\sT = -\delta_i^2 \Up_{i,k}$ and $(x_i)^\tau = \delta_i^2 x_i$; and this will yield the results that $(Q\pr_{k-1}) \sTtau = -Q\pr_{k-1}$ and $(U\pr_{k-1}) \sTtau = -U\pr_{k-1}$ as desired. (Again, this uses the {\it second} characterization of $KO_{-1}(A)$.)

This shows that $[Q_{k-1}\pr] \in \widetilde{KO}_{k-1}(C\pr(S^{k-1}, \tau^{1, k-1}))$ and 
$[U_k\pr] \in \widetilde{KO}_k(C\pr(S^{k}, \tau^{1, k}))$ for all $k$ odd. At this point, the proof that
$$g_{1,k-1} = [Q_{k-1}\pr] \quad \text{and} \quad g_{1,k} = [U_k\pr]$$
 is the same as (the last paragraph of) the proof of Theorem~\ref{main-thmI}.

%
\end{proof}
\newpage


\subsection{{\bf The General Case: $C\pr(S^{d}, \tau^{a,b})$}} \label{sub3}

~

In this subsection, we consider the most general situation of a sphere with an involution (but still excluding the case with the antipodal involution). Let $a$ and $b$ be nonnegative integers where $a \geq 1$ and $a + b = d+1$. Then
$$\widetilde{KO}_*(C\pr( S^d, \tau^{a, b}))  
			= \Sigma^{a-b-1} KO_*(\R) \; $$
and we are looking to identify unitaries that represent the key generators
$$g_{a,b} \in \widetilde{KO}_{b-a+1}(C\pr(S^{d}, \tau^{a,b})) \cong \Z \; .$$

We establish the following notation: for a set $X$ and an integer $i$, we define
$$\delta_{i}^X = \begin{cases} 1 & i \in X \\ -1 & i \notin X . \end{cases} $$

\begin{lemma} \label{lemma:modifyCG}
Let $a$ and $b$ be nonnegative integers where $a + b  = k$ is odd. 
Then there exist a set $X \subseteq \{1,2, \dots, k\}$
that satisfies $|X| = a $, $|X^c| = b$, and a set of Clifford generators
$$\Up_{1,k}, \dots, \Up_{k,k}$$ such that
\begin{enumerate}
\item If $a-b \equiv 1 \pmod 8$, then  
$\Up_{i,k}\T = \delta_{i}^{X} \Up_{1,k}$.
 \item If $a-b \equiv 7 \pmod 8$, then 
$\Up_{i,k}\T = - \delta_{i}^{X} \Up_{1,k}$.
\item If $a-b \equiv 5 \pmod 8$, then 
$\Up_{i,k}\sT = \delta_{i}^{X} \Up_{1,k}$.
 \item If $a-b \equiv 3 \pmod 8$, then 
$\Up_{i,k}\sT = - \delta_{i}^{X} \Up_{1,k}$.
\end{enumerate}
\end{lemma}

\begin{proof}
For the Case (1), assume $a - b\equiv 1 \pmod 8$. By adding an integer multiple of 4 to $a$ and subtracting the same amount from $b$, we obtain integers $a'$ and $b'$ such that
the following hold:
\begin{align*} a' &= b' + 1  &  ~ a & \equiv a'  \pmod 4, \\
 a + b &= a' + b'  
  &  b & \equiv b' \pmod 4 \; .
\end{align*}

Consider the standard set of Clifford generators
$$\Gamma_{1,k}, \dots, \Gamma_{k,k} \; .$$
Then $$\Gamma_{i,k}\T = \delta_i^{X'} \Gamma_{i,k}$$
where $X'$ is the set of positive odd integers less than or equal to $k$. 
Then  
$$|X'| = (k+1)/2  = a' \quad \text{and} \quad |(X')^c| = (k-1)/2  = b'\; .$$
Note that if $a = a'$ and $b = b'$, then we are done with $X = X'$ and $\Up_{i,k} = \Gamma_{i,k}$.

Otherwise, use Lemma~\ref{toggle} to toggle some subset $S$ of these Clifford generators, where the cardinality 
$S$ is a multiple of 4. This subset will correspond either to a set of symmetric Clifford generators that we make antisymmetric; or to a set of antisymmetric Clifford generators that we make symmetric. In this way, we replace $X'$ with a new set $X$ of cardinality $a$, indexing exactly the new Clifford generators that are symmetric. This new set of Clifford generators is labeled as $\Up_{i,k}$ and satisfies Part (1).

For Case (2), assume that $a - b \equiv 7 \pmod 8$.  Again, starting with the standard set of Clifford generators we have
$$\Gamma_{i,k}\T = - \delta_i^{X'} \Gamma_{i,k}$$
where $X'$ is the set of positive even integers less than or equal to $k$.  This time we have 
$$|X'| = (k-1)/2  \quad \text{and} \quad |(X')^c| = (k+1)/2  \; $$
which satisfies $|X'| - |(X')^c| = -1$.
So if $a - b = -1$ we are done, using $X = X'$. Otherwise, we use Lemma~\ref{toggle} to toggle some subset $S$ of these Clifford generators (where $|S| \equiv 0 \pmod 4$) to obtain a set of Clifford generators where the cardinality of the antisymmetric generators is $|X| = a$.

For Case (3), assume that $a  - b \equiv 5 \pmod 8$. With the standard set of Clifford generators
we have 
$$\Gamma_{i,k}\sT =  \delta_i^{X'} \Gamma_{i,k}$$
where $X'$ is the set of odd integers between $5$ and $k$ (inclusive). Then 
$$|X'| = (k-3)/2  \quad \text{and} \quad |(X')^c| = (k+3)/2  \; $$
which satisfies $|X'| - |(X')^c| = -3$.
If $a - b = -3$ we are done,
otherwise we toggle some subset of $X'$ to obtain $X$ and the desired set of Clifford generators.

Finally, for Case (4), assume that $a  - b \equiv 3 \pmod 8$. With the standard set of Clifford generators
we have 
$$\Gamma_{i,k}\sT =  - \delta_i^{X'} \Gamma_{i,k}$$
where 
$$X' = \{1,3\} \cup \{ k \in \{2, 4, \dots, k-1\} \mid \text{$k$ even} \}$$
so we have
$$|X'| = (k+3)/2  \quad \text{and} \quad |(X')^c| = (k-3)/2  \; $$
which satisfies $|X'| - |(X')^c| = 3$.
If $a - b = 3$ we are done, otherwise we again toggle some subset of $X'$ as in the previous cases to obtain the desired set of Clifford generators.
\end{proof}

\vspace{1cm}

\begin{const} \label{construction3}
Let $a$ and $b$ be nonnegative integers where $a + b  = k$ is odd. This construction shows how to produce particular unitaries $Q_{a,b}$ and $U_{a,b+1}$, which ultimately we show to represent the generators of the $KO$-groups indicated, as indicated in Theorem~\ref{main-thm3} below.

\begin{enumerate} 
\item[(1)] Assume $a-b \equiv 1 \pmod 8$
\begin{itemize}
\item Let $\Up_{1,k}, \dots, \Up_{k,k}$ and $X$ be as given in Lemma~\ref{lemma:modifyCG} so that
 $$\Gamma_{i,k} \T =  \delta_i^X \, \Gamma_{i,k}  \; .$$
\item Arrange the variables in $\R^k$ so that the involution $\tau^{a,b}$ on $S^{k-1} \subset \R^{k}$ is given by 
$$x_i \mapsto \delta_i^X x_i \; .$$
\item Extend $\tau^{a,b}$ to an involution $\tau^{a, b+1}$ on $S^{k+1}$ by  $x_{k+1} \mapsto -x_{k+1} \; .$
\item Let $Q\pr_{a,b}(x) = \sum_{i = 1}^k x_i \Gamma_{i,k}$ and $U\pr_{a,b+1}(x) = \sum_{i = 1}^k x_i \Gamma_{i,k} + x_{k+1} i I_n \; .$
\item Then $[Q\pr_{a,b}] \in \widetilde{KO}_{b-a+1}( C\pr(S^{k-1}, \tau^{a,b}))$ 
and $[U\pr_{a,b+1}] \in  \widetilde{KO}_{b-a+2}( C\pr(S^{k}, \tau^{a,b+1}))$.
\end{itemize}

\vspace{.5cm}

\item[(2)] Assume $a-b \equiv 7 \pmod 8$
\begin{itemize}
\item Let $\Up_{1,k}, \dots, \Up_{k,k}$ and $X$ be as given in Lemma~\ref{lemma:modifyCG} so that
 $$\Gamma_{i,k} \T =  - \delta_i^X \, \Gamma_{i,k}  \; .$$
\item Arrange the variables in $\R^k$ so that the involution $\tau^{a,b}$ on $S^{k} \subset \R^{k+1}$ is given by 
$$x_i \mapsto \delta_i^X x_i \; .$$
\item Extend $\tau^{a,b}$ to an involution $\tau^{a, b+1}$ on $S^{k+1}$ by  $x_{k+1} \mapsto -x_{k+1} \; .$
\item Let $Q\pr_{a,b}(x) = \sum_{i = 1}^k x_i \Gamma_{i,k}$ and $U\pr_{a,b+1}(x) = \sum_{i = 1}^k x_i \Gamma_{i,k} + x_{k+1} i I_n \; .$
\item Then $[Q\pr_{a,b}] \in \widetilde{KO}_{b-a+1}( C\pr(S^{k-1}, \tau^{a,b}))$ 
and $[U\pr_{a,b+1}] \in  \widetilde{KO}_{b-a+2}( C\pr(S^{k}, \tau^{a,b+1}))$.
\end{itemize}

\item[(3)] Assume  $a-b \equiv 5 \pmod 8$
\begin{itemize}
\item Let $\Up_{1,k}, \dots, \Up_{k,k}$ and $X$ be as given in Lemma~\ref{lemma:modifyCG} so that
 $$\Gamma_{i,k} \sT =  \delta_i^X \, \Gamma_{i,k}  \; .$$
\item Arrange the variables in $\R^k$ so that the involution $\tau^{a,b}$ on $S^{k} \subset \R^{k+1}$ is given by 
$$x_i \mapsto \delta_i^X x_i \; .$$
\item Extend $\tau^{a,b}$ to an involution $\tau^{a, b+1}$ on $S^{k+1}$ by  $x_{k+1} \mapsto -x_{k+1} \; .$
\item Let $Q\pr_{a,b}(x) = \sum_{i = 1}^k x_i \Gamma_{i,k}$ and $U\pr_{a,b+1}(x) = \sum_{i = 1}^k x_i \Gamma_{i,k} + x_{k+1} i I_n \; .$
\item Then $[Q\pr_{a,b}] \in \widetilde{KO}_{b-a+2}( C\pr(S^{k-1}, \tau^{a,b}))$ 
and $[U\pr_{a,b+1}] \in  \widetilde{KO}_{b-a+2}( C\pr(S^{k}, \tau^{a,b+1}))$.
\end{itemize}

\item[(4)] Assume $a-b \equiv 3 \pmod 8$
\begin{itemize}
\item Let $\Up_{1,k}, \dots, \Up_{k,k}$ and $X$ be as given in Lemma~\ref{lemma:modifyCG} so that
 $$\Gamma_{i,k} \sT =  -\delta_i^X \, \Gamma_{i,k}  \; .$$
\item Arrange the variables in $\R^k$ so that the involution $\tau^{a,b}$ on $S^{k} \subset \R^{k+1}$ is given by 
$$x_i \mapsto \delta_i^X x_i \; .$$
\item Extend $\tau^{a,b}$ to an involution $\tau^{a, b+1}$ on $S^{k+1}$ by  $x_{k+1} \mapsto -x_{k+1} \; .$
\item Let $Q\pr_{a,b}(x) = \sum_{i = 1}^k x_i \Gamma_{i,k}$ and $U\pr_{a,b+1}(x) = \sum_{i = 1}^k x_i \Gamma_{i,k} + x_{k+1} i I_n \; .$
\item Then $[Q\pr_{a,b}] \in \widetilde{KO}_{b-a+1}( C\pr(S^{k-1}, \tau^{a,b}))$ 
and $[U\pr_{a,b+1}] \in  \widetilde{KO}_{b-a+2}( C\pr(S^{k}, \tau^{a,b+1}))$.
\end{itemize}
\end{enumerate}

\vspace{2cm}

\end{const}

\begin{thm} \label{main-thm3}
Let $a$ and $b$ be nonnegative integers where $a + b  = k$ is odd.
The elemenets $[Q\pr_{a,b}]$ and $[U\pr_{a,b+1}]$ from Parts (1)-(4) of Construction~\ref{construction3} represent generators of 
$\widetilde{KO}_{b-a+1}(C\pr(S^{k}, \tau^{a,b+1})) = \Z$ and $\widetilde{KO}_{b-a+2} (C\pr(S^{k+1}, \tau^{a+1,b+1})) = \Z$, respectively.
In other words,
$$g_{a,b} = [Q\pr_{a,b}] \qquad \text{and} \qquad g_{a,b+1} = [U\pr_{a,b+1}] \; .$$
\end{thm}

\begin{proof}
In the first case, $a - b \equiv 1 \pmod 8$,  we can easily check that $Q\pr_{a,b}$ is self-adjoint and also calculate
\begin{align*}
	(Q\pr_{a,b}) \Ttau(x)   
		&= \left( \sum_{i = 1}^k x_i \Gamma_{i,k} \right) \Ttau \\
		&= \sum_{i = 1}^k x_i^\tau \Gamma_{i,k}\T \\
		&= \sum_{i = 1}^k (\delta_i^X)^2 x_i \Gamma_{i,k} \\
		&= (Q\pr_{a,b}) (x)    \; .
\end{align*}
Similarly, check that $(U\pr_{a,b+1}) \Ttau = (U\pr_{a,b+1})^*$.
Thus
$$[Q\pr_{a,b}] \in  \widetilde{KO}_{0}( C\pr(S^{k}, \tau^{a,b})) 
\quad \text{and} \quad 
[U\pr_{a,b+1}] \in  \widetilde{KO}_{1}( C\pr(S^{k+1}, \tau^{a,b+1}))  \; .$$

In the second case, $a - b \equiv 7 \pmod 8$, we calculate
\begin{align*}
	(Q\pr_{a,b}) \Ttau(x)   
		&= \left( \sum_{i = 1}^k x_i \Gamma_{i,k} \right) \Ttau \\
		&= \sum_{i = 1}^k x_i^\tau \Gamma_{i,k}\T \\
		&= \sum_{i = 1}^k (- \delta_i^X)^2 x_i \Gamma_{i,k} \\
		&= - (Q\pr_{a,b}) (x)    
\end{align*}
and similarly we find that $(U\pr_{a,b+1}) \Ttau = -U\pr_{a,b+1}$. 
So
$$[Q\pr_{a,b}] \in  \widetilde{KO}_{2}( C\pr(S^{k}, \tau^{a,b}))  
\quad \text{and} \quad 
[U\pr_{a,b+1}] \in  \widetilde{KO}_{3}( C\pr(S^{k+1}, \tau^{a,b+1}))  \; .$$
(Note that here we are using the second characterization of $KO_3(A)$ from Table~ \ref{unitaryTable}, as represented by unitaries $u$ that satisfy $u\Ttau = -u$.)

In the third case $a - b \equiv 5 \pmod 8$, similar calculations show that 
$$(Q\pr_{a,b}) \sTtau = Q\pr_{a,b} \quad \text{and} \quad (U\pr_{a,b+1}) \sTtau = (U\pr_{a,b+1})^* \;  $$
so 
$$[Q\pr_{a,b}] \in  \widetilde{KO}_{4}( C\pr(S^{k}, \tau^{a,b}))  
\quad \text{and} \quad 
[U\pr_{a,b+1}] \in  \widetilde{KO}_{5}( C\pr(S^{k+1}, \tau^{a,b+1})) \; .$$

And in the fourth case, $a - b \equiv 3 \pmod 8$ we have
$$(Q\pr_{a,b}) \sTtau = -Q\pr_{a,b} \quad \text{and} \quad (U\pr_{a,b+1}) \sTtau = -U\pr_{a,b+1}\;  $$
so 
$$[Q\pr_{a,b}] \in  \widetilde{KO}_{6}( C(S^{k}, \tau^{a,b}))  
\quad \text{and} \quad 
[U\pr_{a,b+1}] \in  \widetilde{KO}_{7}( C(S^{k+1}, \tau^{a,b+1})) \; .$$
(Again, we are using the second characterization of $KO_{-1}(A)$ from Table~ \ref{unitaryTable}, as represented by unitaries $u$ that satisfy $u\sTtau = -u$.)

Therefore in all cases, we have 
$$[Q\pr_{a,b}] \in \widetilde{KO}_{b-a+1}(C(S^{k}, \tau^{a,b})) \quad \text{and} \quad 
[U\pr_{a,b+1}] \in \widetilde{KO}_{b-a+2}(C(S^{k+1}, \tau^{a,b+1})) .$$

Then it follows using the same argument as in the last paragraph of the proof of Theorem~\ref{main-thmI} that 
$$g_{a,b} = [Q\pr_{a,b}] \quad \text{and} \quad g_{a, b+1} = [U\pr_{a,b+1}] \; .$$
\end{proof}

\vspace{2cm}

\bibliographystyle{amsplain}
\bibliography{sphere}

@article{boersema-spheres1,
	author = {Jeffrey L. Boersema},
	doi = {https://doi.org/10.1016/j.exmath.2025.125734},
	issn = {0723-0869},
	journal = {Expositiones Mathematicae},
	number = {2},
	title = {The real {$K$}-theory of the sphere with the antipodal involution},
	url = {https://www.sciencedirect.com/science/article/pii/S0723086925000891},
	volume = {44},
	year = {2026}}

@book{BoersemaSchochet,
	author = {Boersema, Jeffrey and Schochet, Claude L.},
	publisher = {Springer, Heidelberg},
	series = {Lecture Notes in Mathematics},
	title = {Real {$K$}-Theory for {$C \sp *$}-Algebras: Just the Facts},
	year = {Forthcoming}}

@article{AtiyahR,
	author = {Atiyah, M. F.},
	doi = {10.1093/qmath/17.1.367},
	fjournal = {The Quarterly Journal of Mathematics. Oxford. Second Series},
	issn = {0033-5606,1464-3847},
	journal = {Quart. J. Math. Oxford Ser. (2)},
	mrclass = {55.30 (57.30)},
	mrnumber = {206940},
	mrreviewer = {J.\ F.\ Adams},
	pages = {367--386},
	title = {{$K$}-theory and reality},
	url = {https://doi.org/10.1093/qmath/17.1.367},
	volume = {17},
	year = {1966}}

@article{Boer2020,
	author = {Boersema, Jeffrey L.},
	fjournal = {Houston Journal of Mathematics},
	issn = {0362-1588},
	journal = {Houston J. Math.},
	mrclass = {46L80 (19K99 46L60 46L85 81R10)},
	mrnumber = {4137279},
	mrreviewer = {Michael Frank},
	number = {1},
	pages = {71--111},
	title = {{$K$}-theory for real {$C^*$}-algebras via unitary elements with symmetries, {P}art {II}---{N}atural transformations and {$KO_*(\Bbb R)$}-module operations},
	volume = {46},
	year = {2020}}

@article{Boer2002,
	author = {Boersema, Jeffrey L.},
	coden = {KTHEEO},
	fjournal = {$K$-Theory. An Interdisciplinary Journal for the Development, Application, and Influence of $K$-Theory in the Mathematical Sciences},
	issn = {0920-3036},
	journal = {$K$-Theory},
	mrclass = {46L05 (19K35 46L80 46M18)},
	mrnumber = {1935138},
	mrreviewer = {Judith A. Packer},
	number = {4},
	pages = {345--402},
	title = {Real {$C\sp *$}-algebras, united {$K$}-theory, and the {K}\"unneth formula},
	volume = {26},
	year = {2002}}

@article{BL,
	author = {Boersema, Jeffrey L. and Loring, Terry A.},
	fjournal = {New York Journal of Mathematics},
	journal = {New York J. Math.},
	pages = {1139--1220},
	title = {{$K$}-theory for real {$C^*$}-algebras via unitary elements with symmetries},
	volume = {22},
	year = {2016}}

@article{Li...2025,
	author = {Ki Young Lee and Stephan Wong and Sachin Vaidya and Terry Loring and Alexander Cerjan},
	journal = {Phys. Rev. Research},
	title = {Classification of fragile topology enabled by matrix homotopy},
	volume = {7},
	year = {2025}}

@article{SBS,
	author = {Schulz-Baldes, Hermann and Stoiber, Tom},
	doi = {10.1016/j.exmath.2023.125519},
	fjournal = {Expositiones Mathematicae},
	issn = {0723-0869},
	journal = {Expo. Math.},
	mrclass = {19E99 (19D55)},
	mrnumber = {4667589},
	mrreviewer = {Jonathan M. Rosenberg},
	number = {4},
	pages = {Paper No. 125519, 13},
	title = {The generators of the {$K$}-groups of the sphere},
	url = {https://doi.org/10.1016/j.exmath.2023.125519},
	volume = {41},
	year = {2023}}

@article{CL-2022,
	author = {Alexander Cerjan and Terry A. Loring},
	journal = {Phys. Rev. B},
	month = {August},
	title = {Local invariants identify topology in metals and gapless systems},
	year = {2022}}

@article{CLSB-2024,
	author = {Alexander Cerjan and Terry A Loring and Hermann Schulz-Baldes.},
	journal = {Phys. Rev. Lett.},
	title = {Local Markers for crystalline topology},
	volume = {132},
	year = {2024}}

@article{DLC-2023,
	author = {Kahlil Y. Dixon and Terry A. Loring and Alexander Cerjan},
	journal = {Phys. Rev. Lett.},
	month = {November},
	title = {Classifying Topology in PHotonic heterostructures with Gapless Environments},
	volume = {131},
	year = {2023}}

@article{CL2024,
	author = {Alexander Cerjan and Terry A Loring},
	journal = {APL Photonics},
	month = {November},
	title = {Classying photonic topology using the spectral localizer and numerical {$K$}-theory},
	volume = {9},
	year = {2024}}

@article{Loring2015,
	author = {Terry A. Loring},
	journal = {Annals of Physics},
	month = {May},
	pages = {383--416},
	title = {{$K$}-theory and pseudospectra for topological insulators},
	volume = {356},
	year = {2015}}

@article{bradlyn,
	author = {B. Bradlyn et al.},
	journal = {Nature},
	pages = {298--305},
	title = {Topological quantum chemistry},
	volume = {547},
	year = {2017}}

@article{bousfield90,
	author = {Bousfield, A. K.},
	coden = {JPAAA2},
	fjournal = {Journal of Pure and Applied Algebra},
	issn = {0022-4049},
	journal = {J. Pure Appl. Algebra},
	mrclass = {55N15 (19L99 55P42)},
	mrnumber = {1075335 (92d:55003)},
	mrreviewer = {Andrei V. Pazhitnov},
	number = {2},
	pages = {121--163},
	title = {A classification of {$K$}-local spectra},
	volume = {66},
	year = {1990}}

@article{HL,
	author = {Hastings, Matthew B. and Loring, Terry A.},
	coden = {ADNYA6},
	doi = {10.1016/j.aop.2010.12.013},
	fjournal = {Annals of Physics},
	issn = {0003-4916},
	journal = {Ann. Physics},
	mrclass = {82B10 (19K99 46L05 81R15)},
	mrnumber = {2806137 (2012j:82005)},
	mrreviewer = {Albert Jeu-Liang Sheu},
	number = {7},
	pages = {1699--1759},
	title = {Topological insulators and {$C^*$}-algebras: theory and numerical practice},
	url = {http://dx.doi.org/10.1016/j.aop.2010.12.013},
	volume = {326},
	year = {2011}}

\end{document}